\documentclass[a4paper,leqno]{amsart}

\usepackage{amsmath}
\usepackage{amsthm}
\usepackage{amssymb}
\usepackage{amsfonts}
\usepackage{graphicx} 
\usepackage[boxruled,norelsize]{algorithm2e}
\usepackage{subfig}
\usepackage{color}

\numberwithin{equation}{section}

\def\C{{\mathbb C}}% complex numbers
\def\R{{\mathbb R}}% real numbers
\def\N{{\mathbb N}}% nonnegative integers
\def\le{\leqslant}
\def\ge{\geqslant}

\newcommand{\re}{\mathrm{Re}}
\newcommand{\im}{\mathrm{Im}}

\newcommand{\conj}[1]{\overline{#1}}

\newcommand{\pare}[1]{\left(#1\right)}

\theoremstyle{plain}
\newtheorem{theorem}{Theorem}[section]
\newtheorem*{theorem*}{Theorem}
\newtheorem{lemma}[theorem]{Lemma}
\newtheorem{corollary}[theorem]{Corollary}
\newtheorem{proposition}[theorem]{Proposition}

\theoremstyle{definition}

\newtheorem{remark}[theorem]{Remark}
\newtheorem*{remark*}{Remark}

\begin{document}
\title[Optimal bilinear control of Gross--Pitaevskii equations]
{Optimal bilinear control of Gross--Pitaevskii equations}
\author[M.\ Hinterm\"uller]{Michael Hinterm\"uller}
\address[M.\ Hinterm\"uller]{Department of Mathematics, 
Humboldt--Universit\"at zu Berlin,
Unter den Linden 6,
D--10099 Berlin, Germany}
\email{hint@mathematik.hu-berlin.de}
\author[D.\ Marahrens]{Daniel Marahrens}
\address[D.\ Marahrens]{Department of Applied Mathematics and Theoretical
Physics\\
CMS, Wilberforce Road\\ Cambridge CB3 0WA\\ England}
\email{d.o.j.marahrens@damtp.cam.ac.uk}
\author[P.\ A.\ Markowich]{Peter A.\ Markowich}
\address[P.\ A.\ Markowich]{Department of Applied Mathematics and Theoretical
Physics\\
CMS, Wilberforce Road\\ Cambridge CB3 0WA\\ England}
\email{p.a.markowich@damtp.cam.ac.uk}
\author[C.\ Sparber]{Christof Sparber}
\address[C.\ Sparber]{Department of Mathematics, Statistics, and Computer Science,
University of Illinois at Chicago,
851 South Morgan Street
Chicago, Illinois 60607, USA}
\email{sparber@uic.edu}

\begin{abstract}
A mathematical framework for optimal bilinear control of nonlinear Schr\"odinger equations of Gross--Pitaevskii type arising in the description of Bose--Einstein condensates is presented. The obtained results generalize earlier efforts found in the literature in several aspects. In particular, the cost induced by the physical work load over the control process is taken into account rather then often used $L^2$- or $H^1$-norms for the cost of the control action. 
Well-posedness of the problem and existence of an optimal control is proven. In addition, the first order optimality system is rigorously derived. 
Also a numerical solution method is proposed, which is based on a Newton type iteration, and used to solve several coherent quantum control problems.
\end{abstract}

\date{\today}


\thanks{This publication is based on work supported by Award No.\ KUK-I1-007-43, funded by the King Abdullah University of Science and Technology (KAUST). M.H. also acknowledges support through the START-Project Y305 ''Interfaces and Free Boundaries'' administered by the Austrian Science Fund FWF, the German Science Fund DFG under SPP1253 "Optimization with PDE Constraints'' and the DFG-Research Center {\sc MATHEON} under projects C28 and C31.}
\maketitle

\section{Introduction}\label{sec:intro}

\subsection{Physics background} Ever since the first experimental realization of \emph{Bose--Einstein condensates} (BECs) in 1995, the possibility to store, manipulate, and 
measure a single quantum system with extremely high precision has provided great stimulus in many fields of physical and mathematical research, 
among them {\it quantum control theory}. In the regime of dilute gases, a BEC, consisting of $N$ particles, can be modeled by the 
 \emph{Gross--Pitaevskii equation} \cite{PiSt}, i.e.\ a cubically nonlinear Schr\"odinger equation (NLS) of the form
\[
 i \hbar \partial_t \psi = -\frac{\hbar^2}{2m}\Delta \psi + U(x)\psi + N g |\psi|^{2}\psi + W(t, x) \psi, \quad x\in \R^3, t\in \R,
\]
with $m$ denoting the mass of the particles, $\hbar$ Planck's constant, $g = 4\pi \hbar^2 a_{\rm sc}/m$, and $a_{\rm sc}\in \R$ their characteristic scattering length, describing the inter-particle collisions. 
The function $U(x)$ describes an external trapping potential which is necessary for the experimental realization of a BEC. Typically, $U(x)$ is assumed to be a {\it harmonic confinement}. 
In situations where $U(x)$ is strongly anisotropic, one experimentally obtains a quasi one-dimensional (``cigar-shaped"), or quasi two-dimensional (``pancake shaped") BEC, see for instance \cite{KNSQ}. 
In the following, we shall assume $U(x)$ to be {\it fixed}. The condensate is consequently manipulated via a time-dependent \emph{control potential} ${W} (t,x)$, which 
we shall assume to be of the following form: $${W}(t,x) = \alpha(t) V(x),$$
Here, $\alpha(t)$ denotes the {\it control parameter} (typically, a switching function acting within a certain time-interval $[0,T]$) and $V(x)$ is a given potential. 
In our context, the potential $V(x)$ models the spatial profile of a laser field used to manipulate the BEC and $\alpha(t)$ its intensity.

The problem of quantum control, i.e. the coherent manipulation of quantum systems (in particular Bose--Einstein condensates) 
via external potentials $W(t,x)$, has attracted considerable interest in the physics literature, cf.\ \cite{BVR, CB, schm, Ho, PDR, RKP, ZR}.
From the mathematical point of view, quantum control problems are a specific example of \emph{bilinear control systems} \cite{Cor}. 
It is known that linear or nonlinear Schr\"odinger--type equations are in general \emph{not exactly controllable} in, say, $L^2(\R^3)$, cf.\ \cite{Tu}. 
Similarly, approximate controllability is known to hold for only some specific systems, such as \cite{MR}. 
More recently, however, sufficient conditions for approximate controllability of linear Schr\"odinger equations with purely discrete spectrum have been derived in \cite{CMSB}. 
In \cite{MS} these conditions have been shown to be generically satisfied, but, to the best of our knowledge, a generalization to the case of nonlinear Schr\"odinger equations is still lacking. 

The goal of the current paper is to consider quantum control systems within the 
framework of \emph{optimal control}, cf.\ \cite{WG} for a general introduction, from a partial differential equation constrained point of view. 
The objective of the control process is thereby quantified through an objective functional $J=J(\psi, \alpha)$, which is minimized subject to the condition that the time-evolution of the 
quantum state is governed by the Gross--Pitaevskii equation. Such objective functionals $J(\psi, \alpha)$ usually consist of two parts, one being the desired physical quantity 
(observable) to be minimized, the other one describing the cost it takes to obtain the desired outcome through the control process. In quantum mechanics, the wave function 
$\psi (t,\cdot)$ itself is not a physical observable. Rather, one considers self-adjoint linear operators $A$ acting on $\psi(t,\cdot)$ and 
aims for a prescribed \emph{expectation value} of $A$ at time $t=T>0$, the final time of the control process. 
Such expectation values are computed by taking the $L^2$--inner product $\langle \psi(T,\cdot), A \psi(T,\cdot) \rangle_{L^2(\R^d)}$. 
Note that this implies that the corresponding $\psi(t,\cdot)$ is only determined up to a constant phase. This fact makes quantum control less  ``rigid" when 
compared to classical control problems in which one usually aims to optimize for a prescribed target state.

There are many possible ways of modeling the cost it takes to reach a certain prescribed expectation value. The corresponding cost terms within $J(\psi, \alpha)$ are often given 
by the norm of the control $\alpha(t)$ in some function space. Typical choices are $L^2(0,T)$ or $H^1(0,T)$. 
However, these choices of function spaces for $\alpha(t)$ often lack a clear physical interpretation. In addition, cost terms based on, say, the $L^2$--norm of $\alpha$ tend to yield 
highly oscillatory optimal controls due to the oscillatory nature of the underlying (nonlinear) Schr\"odinger equation. The same is true for quantum control via so-called Lyapunov tracking methods, see, e.g., \cite{CGLT}.
In the present work we shall present a novel choice for the cost term, which is based on the corresponding 
physical work performed throughout the control process. 

We continue this introductory section by describing the mathematical setting in more detail.

\subsection{Mathematical setting}
We consider a quantum mechanical system described by a wave function $\psi(t,\cdot) \in L^2(\R^d)$ within $d=1,2,3$ spatial dimensions. The case $d=1,2$ models 
the effective dynamics within strongly anisotropic potentials (resulting in a quasi one or two-dimensional BEC).
The time-evolution of $\psi(t,\cdot)$ is governed by the following generalized Gross--Pitaevskii equation (rescaled into dimensionless form):
\begin{equation}\label{eq:schroed}
 i \partial_t \psi = -\frac{1}{2}\Delta \psi  + U(x) \psi+ \lambda |\psi|^{2\sigma}\psi + \alpha(t) V(x) \psi, \quad x\in \R^d, t \in \R,
\end{equation}
with $\lambda \ge 0$, $\sigma<2/(d-2)$, and subject to initial data
\[
 \psi(0,\cdot) = \psi_0\in L^2(\R^d), \quad \alpha(0) = \alpha_0\in \R.
\]
For physical reasons we normalize $\| \psi_0 \|_{L^2(\R^d)} = 1$, which is henceforth preserved by the time-evolution of \eqref{eq:schroed}. 
In addition, the control potential is assumed to be $V\in W^{1,\infty}(\R^d)$, whereas for $U(x)$ we require
\[
 U\in C^\infty(\R^d) \text{ such that } \partial^k U \in L^\infty(\R^d) \text{ for all multi-indices $k$ with } |k|\ge 2.
\]
In other words, the external potential is assumed to be smooth and \emph{subquadratic}.  One of the most important examples is the harmonic oscillator $U(x) = \frac{1}{2}{|x|}^2$.
Due to the presence of a subquadratic potential, we restrict ourselves to initial data $\psi_0$ in the \emph{energy space}
\begin{equation}
\label{eq:en_space}
 \Sigma: = \left\{ \psi\in H^1(\R^d) \ : \ x\psi \in L^2(\R^d) \right\}.
\end{equation}
In particular, this definition guarantees that the quantum mechanical energy functional
\begin{equation}\label{eq:energy}
 E(t) = \int_{\R^d} \frac{1}{2} |\nabla \psi(t,x)|^2 + \frac{\lambda}{\sigma+1} |\psi(t,x)|^{2\sigma+2} + (\alpha(t) V(x) + U(x)) |\psi(t,x)|^2 dx
\end{equation}
associated to~\eqref{eq:schroed}  is well defined.

\begin{remark} Note that $\sigma<2/(d-2)$ allows for general power law nonlinearities in dimensions $d=1,2$, whereas in $d=3$ the nonlinearity is 
assumed to be less than quintic. From the physics point of view a cubic nonlinearity $\sigma =1$ is the most natural choice, but higher order nonlinearities also 
arise in systems with more complicated inter-particle interactions, in particular in lower dimensions; compare \cite{KNSQ}. 
From the mathematical point of view, it is well known that the restriction $\sigma<2/(d-2)$ guarantees well-posedness of the initial value problem in the energy space $\Sigma$; see \cite{car09, caze}. 
In addition, the condition $\lambda\ge0$ (defocusing nonlinearity) guarantees the existence of global in-time solutions to \eqref{eq:schroed}; see \cite{car09}. Hence, we do not encounter the problem of 
finite-time blow-up in our work.
\end{remark}

Although \eqref{eq:schroed} conserves mass, i.e.\ $\|\psi(t,\cdot)\|_{L^2(\R^d)} = \|\psi_0\|_{L^2(\R^d)}$ for all $t\in \R$, the energy $E(t)$ is not conserved. This is in contrast to the case of time-independent potentials. In our case, rather one finds that 
\begin{equation}\label{eq:Ederiv}
 \frac{d}{dt} E(t) = \dot{\alpha}(t) \int_{\R^d} V(x) |\psi(t,x)|^2 dx.
\end{equation}
The \emph{physical work} performed by the system within a given time-interval $[0,T]$ is therefore equal to
\begin{equation}\label{eq:work}
E(T) - E(0) = \int_0^T \dot{\alpha}(t) \int_{\R^d} V(x) |\psi(t,x)|^2 dx \, dt.
\end{equation}
Thus a control $\alpha(t)$ acting for $t\in [0,T]$ upon a system described by \eqref{eq:schroed} requires a certain amount of energy, which is given by \eqref{eq:work}. 
It, thus, seems natural to include such a term in the cost functional of our problem in order to quantify the control action. 

Indeed, for any given final control time $T>0$, and parameters $\gamma_1\ge0, \gamma_2>0$, we define the following \emph{objective functional}:
\begin{equation}\label{eq:J1}
%\begin{split}
 J(\psi,\alpha) :=  \langle \psi(T,\cdot), A \psi(T,\cdot) \rangle_{L^2(\R^d)}^2 + \gamma_1 \int_0^T (\dot{E}(t))^2 \;dt + \gamma_2 \int_0^T (\dot{\alpha}(t))^2 \;dt,
% \end{split}
\end{equation}
where $A: \Sigma \to L^2(\R^d)$ is a bounded linear operator which is assumed to be essentially self-adjoint on $L^2(\R^d)$. In other words, $A$ represents a physical observable with 
$\text{spec} \,(A)\subseteq \R$. A typical choice for $A$ would be $A = A' - a$ where $a\in \R$ is some prescribed expectation value for the observable $A'$ in the state $\psi(T,x)$. 
For example, if $a\in \text{spec}\, (A')$ is chosen to be an eigenvalue of $A'$, the first term in $J(\psi, \alpha)$ is zero as soon as the target state $\psi(T,\cdot)$ is, up to a phase factor, given by 
an associated eigenfunction of $A'$. 
However, one may consider choosing $a\in \R$ such that it ``forces" the functional to equidistribute between, say, two eigenfunctions. 
\begin{remark} We also remark that in the case $A=P_\varphi - 1$, where $P_\varphi$ denotes the orthogonal projection onto a given \emph{target state} $\varphi\in L^2(\R^d)$, the first term on the right hand side of \eqref{eq:J1} reads 
\begin{equation}\label{eq:AP}  \langle \psi(T,\cdot), A \psi(T,\cdot) \rangle_{L^2(\R^d)} =  \big|\langle \psi(T,\cdot), \varphi(\cdot)\rangle_{L^2(\R^d)} \big|^2 - 1 ,
\end{equation}
using the fact that $\| \psi(T,\cdot) \|_{L^2(\R^d)} = 1$. Expression \eqref{eq:AP} is the same as used in recent works in the physics literature; see \cite{schm}. \end{remark}

Using \eqref{eq:Ederiv}, we find that the objective functional $J(\psi, \alpha)$ explicitly reads
\begin{align}\label{eq:J2}
\begin{split}
 J(\psi,\alpha) := &\ \langle \psi(T,\cdot), A \psi(T,\cdot) \rangle_{L^2(\R^d)}^2 \\
 & + \gamma_1 \int_0^T (\dot{\alpha}(t))^2 \left (\int_{\R^d} V(x) |\psi(t,x)|^2 \;dx \right)^2\; dt + \gamma_2 \int_0^T (\dot{\alpha}(t))^2 \;dt.
\end{split}
\end{align}
Here, the second line on the right hand side displays two cost (or penalization) terms for the control: 
The first one, involving $\gamma_1\ge 0$, is given by the square of the physical work, i.e.\ the right hand side of \eqref{eq:Ederiv}. 
The second is a classical cost term as used in \cite{schm}. In our case, the second term is required as a mathematical \emph{regularization} 
of the optimal control problem, since for general (sign changing) potentials $V\in L^\infty(\R^d)$ the weight factor 
\begin{equation}\label{eq:weight} \omega(t):= \int_{\R^d} V(x) |\psi(t,x)|^2 dx
\end{equation} might vanish for some $t\in \R$. 
In such a situation, the boundedness of variations of $\alpha(t)$ is in jeopardy and the optimal control problem lacks well-posedness. Hence, we require $\gamma_2 >0$ for our mathematical analysis, but typically take $\gamma_2 \ll  \gamma_1$ in our numerics
in Section \ref{sec:num} to keep its influence small. Note, however, that in the case where 
the control potential satisfies the positivity condition $$V(x)\ge \delta >0 \;\; \forall x \in \R^d,$$ 
we may choose $\gamma_2=0$ and all of our results remain valid.

\begin{remark}
In situations where the above positivity condition on $V(x)$ does not hold, one might think of performing a time-dependent \emph{gauge transform} of $\psi$, i.e.\
\[
\tilde{\psi}(t,x) = \exp{\left ( - i \kappa \int_0^t \alpha(s) \;ds \right)} \psi(t,x), 
\]
with a constant $\kappa>\text{min}_{x\in \R^d} V(x)$, assuming that the minimum exists. This yields a Gross--Pitaevskii equation for the wave function $\tilde \psi$ with 
modified control potential $\tilde V(x)= (\kappa +V(x))  > 0 $ for all $x \in \R^d$.
Note, however, that this gauge transform leaves the expression \eqref{eq:J2} unchanged and hence does not improve the stuation.
Only if one also changes the potential $V(x)$ within $J(\psi, \alpha)$ into $\tilde V(x)$, the problem does not require any regularization term (proportional to $\gamma_2$). Note, however, that such a modification yields a 
control system which is no longer (mathematically) equivalent to the original problem. In fact, replacing $V(x)$ by $\tilde V(x)$ in the objective functional $J(\psi,\alpha)$ corresponds to increasing the parameter $\gamma_2$ by $\kappa$.
\end{remark}

\subsection{Relation to other works and organization of the paper}

The mathematical research field of optimal bilinear control of systems governed by partial differential equations is by now classical, cf.\ \cite{F, Li} for a general overview. 
Surprisingly, rigorous mathematical work on optimal (bilinear) control of quantum systems appears very limited, despite the physical significance of the involved applications (cf.\ the references given above). Results on simplified situations, as, e.g., for finite dimensional quantum systems, can be found in \cite{BCG} (see also the references therein). More recently, optimal control problems for linear Schr\"odinger equations have been studied in \cite{BKP, BSS, IK}. In addition, numerical questions related to quantum control are studied in \cite{BS, YKY}. 
Among these papers, the work in \cite{IK} appears closest to our effort. Indeed, in \cite{IK}, the authors provide a framework for bilinear optimal control of 
abstract (linear) Schr\"odinger equations. The considered objective functional involves a cost term proportional to the $L^2$--norm of the control parameter $\alpha(t)$. 
The present work goes beyond the results obtained in \cite{IK} in several repects: First, we generalize the cost functional 
to account for oscillations in $\alpha(t)$ and in particular for the physical work load performed throughout the control process. In addition, we allow for observables $A$ which are unbounded operators on $L^2$.
Second, we consider nonlinear Schr\"odinger equations of Gross--Pitaevskii type, including unbounded (subquadratic) potentials, which are highly significant in the quantum control of BECs.
This type of equation makes the study of the associated control problem considerably more involved from a mathematical point of view.

The rest of this work is organized as follows. 
In section~\ref{sec:existence} we clarify existence of a minimizer for our control problem. In particular, we prove that the corresponding optimal solution $\psi_*(t,x)$ is indeed a mild (and not only a weak) solution of \eqref{eq:schroed}, depending continuously 
on the initial data $\psi_0$. Then, in section~\ref{sec:adjoint} the adjoint equation is derived and analyzed with respect to existence and uniqueness of a solution. It is our primary tool for the description of the derivative of the objective function reduced onto the control space through considering the solution of the Gross-Pitaevskii equation as a function of the control variable $\alpha$. The results of section~\ref{sec:adjoint} are paramount for the derivation of the first order optimality system in section~\ref{sec:derivative}. In section~\ref{sec:num} a gradient- and a Newton-type descent method are defined, respectively, and then used for computing numerical solutions for several illustrative quantum control problems. In particular, we consider the optimal shifting of a linear wave package, splitting of a linear wave package and splitting of a BEC. The paper ends with conclusions on our findings in section~\ref{sec:concl}.

{\bf Notation.} Throughout this work we shall denote strong convergence of a sequence $(x_n)_{n\in\N}$ by $x_n\rightarrow x$ and weak convergence by $x_n\rightharpoonup x$. 
For simplicity, we shall often write $\psi(t)\equiv \psi(t, \cdot)$ and also use the shorthand notation 
$L^p_t L^q_x$ instead of $L^p(0,T;L^q(\R^d))$. Similarly, $H^{1}_t$ stands for $H^1(0,T)$, with dual $H^{-1}_t \equiv (H^1(0,T))^*$.

\section{Existence of minimizers}\label{sec:existence}

We start by specifying the basic functional analytic framework. For any given $T>0$, we consider $ H^1(0,T)$ as the real vector space of control parameters $\alpha(t)\in \R$.
It is known \cite{caze} that for every $\alpha\in H^1(0,T)$, there exists a unique {\it mild solution} $\psi \in C([0,T]; \Sigma)$ of the Gross-Pitaevskii equation. More precisely, $\psi$ solves
\begin{equation*}
 \psi(t,x) = S(t)\psi_0(x) - i \int_0^t S(t-s) \left( \lambda |\psi(s,\cdot)|^{2\sigma}\psi(s,\cdot) + \alpha(s) V \psi(s,\cdot) \right)(x) ds,
\end{equation*}
where from now on we denote by
\begin{equation}\label{eq:schg} 
S(t) =  e^{ -itH} , \quad H = -\frac{1}{2} \Delta + U(x),
\end{equation}
the group of unitary operators $\{S(t)\}_{t\in \R}$ generated by the Hamiltonian $H$. In other words, $S(t)$ describes the time-evolution of the linear, uncontrolled system. 
Next, we define
\begin{equation}
\label{eq:W_0(T)}
\Upsilon(0,T) := L^2(0,T;\Sigma)\cap H^1(0,T;\Sigma^*),
\end{equation}
where $\Sigma^*$ is the dual of the energy space $\Sigma$. Then the appropriate space for our minimization problem is
\[
\Lambda(0,T):= \{ (\psi, \alpha) \in \Upsilon(0,T) \times H^1(0,T) \ \mbox{: $\psi$ is a mild solution of  \eqref{eq:schroed} }\}.
\]
Since the control $\alpha$ is real-valued, it is natural to consider $\Lambda(0,T)$ as a \emph{real vector space} 
and we shall henceforth equip $L^2(\R^d)$ with the scalar product
\begin{equation}
\label{eq:inner_prod}
 \langle \xi, \psi \rangle_{L^2(\R^d)} = \re \int_{\R^d} \xi(x) \conj{\psi(x)} \, dx,
\end{equation}
which is subsequently inherited by all $L^2$-based Sobolev spaces. (Note that 
this choice is also used in \cite{caze}.) From what is said above, we infer that the space $\Lambda(0,T)$ is indeed nonempty. 

With these definitions at hand, the optimal control problem under investigation is to find
\begin{equation}\label{eq:minprob}
 J_* = \inf_{(\psi,\alpha)\in \Lambda(0,T)} J(\psi,\alpha).
 \end{equation}
We are now in the position to state the first main result of this work.
\begin{theorem}
\label{thm:ex_min}
Let $\lambda\ge 0$, $0<\sigma<2/(d-2)$, $V\in W^{1,\infty}(\R^d)$, and $U\in C^\infty(\R^d)$ be subquadratic. Then, for any $T>0$, any initial data $\psi_0 \in \Sigma$, $\alpha_0 \in \R$ and any choice of parameters $\gamma_1\ge0$, $\gamma_2>0$ 
the optimal control problem~\eqref{eq:minprob} has a minimizer $(\psi_*, \alpha_*)\in \Lambda(0,T)$.
\end{theorem}
The proof of this theorem will be split into three steps: In subsection \ref{sec:conv} we shall first prove a convergence result for minimizing, or more precisely, infimizing sequences. We consequently deduce in subsection~\ref{sec:mild} 
that the obtained limit $\psi_*$ is indeed a mild solution of \eqref{eq:schroed}. Finally, we shall prove lower semicontinuity of $J(\psi, \alpha)$ with respect to the convergence obtained before.

\subsection{Convergence of infimizing sequences}\label{sec:conv}

First note that there exists at least one infimizing sequence with an infimum $-\infty \le J_* < +\infty$, since $\Lambda(0,T)\neq \emptyset$ and $J:\Lambda(0,T)\rightarrow \R$. Then we have the following result for any infimizing sequence.
\begin{proposition}
\label{prop:min_limit}
Let $(\psi_n,\alpha_n)_{n\in\N}$ be an infimizing sequence of the optimal control problem given by~\eqref{eq:J1}. Then under the assumptions of Theorem~\ref{thm:ex_min} there exist a subsequence, still denoted by $(\psi_n,\alpha_n)_{n\in\N}$, and functions $\alpha_* \in H^1(0,T)$, $\psi_*\in L^\infty(0,T;\Sigma)$, such that
\begin{align*}
 \alpha_n &\rightharpoonup \alpha_* \text{ in } H^1(0,T), \text{and } \alpha_n \rightarrow \alpha_* \text{ in } L^2(0,T) ,\\
 \psi_n &\rightharpoonup \psi_* \text{ in } L^2(0,T;\Sigma),\\
 \psi_n &\rightarrow \psi_* \text{ in } L^2(0,T;L^2(\R^d)) \cap L^2(0,T;L^{2\sigma+2}(\R^d)),
\end{align*}
as $n \to +\infty$.
Furthermore it holds that
\begin{equation}
\label{eq:loc_sigma_conv}
 \psi_n(t) \rightarrow \psi_*(t) \text{ in } L^2(\R^d), \text{ and } \psi_n(t) \rightharpoonup \psi_*(t) \text{ in } \Sigma
\end{equation}
for almost all $t\in[0,T]$.
\end{proposition}
\begin{proof}  By definition, $J \ge 0$ and thus it is bounded from below. For an infimizing sequence $(\psi_n, \alpha_n)_{n\in \N}$ the sequence of objective functional values $(J(\psi_n,\alpha_n))_{n\in\N}$ 
converges and is bounded on $\R$. Hence, it holds that $J(\psi_n,\alpha_n) \le C < +\infty$ for all $n\in\N$. Since $\gamma_2>0$ it follows that
\[
 \int_0^T (\dot{\alpha}_n(t))^2 \;dt \le C <+\infty.
\]
For smooth $\alpha_n:[0,T]\rightarrow \R$ we compute
\[
 \alpha_n(t) = \alpha_n(0) + \int_0^t \dot{\alpha}_n(s) \;ds \le \alpha_n(0) + \left(T \int_0^T (\dot{\alpha}_n(s))^2 \;ds\right)^{1/2} <+\infty,
\]
and thus $\alpha_n$ is bounded in $L^\infty(0,T)$. By approximation (using the fact that $\alpha_n(0) = \alpha_0$ is fixed), the sequence $(\alpha_n)_{n\in\N}$ is 
uniformly bounded in $L^\infty(0,T)$, which in turn implies a uniform bound in $L^2(0,T)$ and thus in $H^1(0,T)$. Hence, 
there exists a subsequence, still denoted $(\alpha_n)_{n\in\N}$, and $\alpha_* \in H^1(0,T)$, such that
\[
 \alpha_n \rightharpoonup \alpha_* \in H^1(0,T).
\]
Moreover, since $H^1(0,T)$ is compactly embedded into $L^2(0,T)$, we deduce that $\alpha_n \rightarrow \alpha_*$ in $L^2(0,T)$. Next, we recall that 
\[
 \frac{d}{dt} E_n(t) = \dot{\alpha}_n(t) \int_{\R^d} V(x) |\psi_n(t,x)|^2 \;dx
\]
and hence
\[
\| \dot E_n \|_{L^2_t} \le \|\dot{\alpha}_n \|_{L^2_t} \| V\|_{L^\infty_x} \| \psi_0 \|_{L^2_x}^2,
\]
in view of mass conservation $\| \psi_n(t) \|_{L^2_x} = \| \psi_0 \|_{L^2_x}$. Since $E_n(0)=E_0$ depends only on $\psi_0$ and $\alpha_0$ (and is thus independent of $n\in \N$), 
the same argument as before yields $\|E_n\|_{L^\infty_t} \le C$. Recalling the definition of the energy \eqref{eq:energy} and the fact that $\lambda \ge 0$, we obtain
\begin{equation}
 \label{eq:nabla_bound}
 \frac{1}{2}\|\nabla \psi_n(t)\|_{L^2_x}^2 \le \|E_n\|_{L^\infty_t} + c \|\alpha_n\|_{L^\infty_t} \|\psi_0\|_{L^2_x}^2 + C\|x\psi_n(t)\|_{L^2_x}^2,
\end{equation}
again using conservation of mass $\|\psi_n(t)\|_{L^2_x} = \|\psi_0\|_{L^2_x}$. Furthermore, it holds that
\begin{align*}
 \frac{d}{dt} \int_{\R^d} |x|^2 |\psi|^2 \; dx &= 2\, \re \int_{\R^d} i |x|^2 \conj{\psi} \left( \frac{1}{2}\Delta \psi - \lambda |\psi|^{2\sigma}\psi - \lambda \alpha V \psi - U\psi \right) dx\\
 &= 2\, \im \int_{\R^d} x\conj{\psi} \nabla\psi \;dx \le \|x\psi(t)\|_{L^2_x}^2 + \|\nabla \psi(t)\|_{L^2_x}^2,
\end{align*}
which, in view of the bound~\eqref{eq:nabla_bound} and Gronwall's inequality, yields
\[
 \|x\psi(t)\|_{L^2_x}^2 \le C\left(\|E_n\|_{L^\infty_t} + \|\alpha_n\|_{L^\infty_t} \|\psi_0\|_{L^2_x}^2\right),
\]
for all $t\in [0,T]$.
In summary we have shown
\begin{equation} \label{eq:H1_bounds}
 \|\psi_n(t)\|_{\Sigma}^2 = \|\psi_n(t)\|_{H^1_x}^2 + \|x\psi_n(t)\|_{L^2_x}^2 \le C,
\end{equation}
where $C>0$ is independent of $n\in\N$ and $t\in[0,T]$. Hence, $\psi_n$ is uniformly bounded in $L^\infty(0,T;\Sigma)$ and in particular in $L^2(0,T; \Sigma)$. 
By reflexivity of $L^2(0,T;\Sigma)$, we consequently infer the existence of a subsequence (denoted by the same symbol) such that
\[
 \psi_n \rightharpoonup \psi_* \text{ in } L^2(0,T;\Sigma)\quad \text{as $n\to + \infty$}.
\]
To obtain the strong convergence announced above, we first note that \eqref{eq:schroed} implies $\partial_t \psi_n \in L^\infty(0,T;\Sigma^*)$. On the other hand, $\Sigma$ is compactly embedded in $L^2(\R^d)$. Thus, we can apply the Aubin--Lions Lemma to deduce 
\[
 \psi_n \xrightarrow{n \to \infty} \psi_* \text{ in } L^2((0,T)\times\R^d).
\]\color{black}{}
In particular, there exists yet another subsequence (still denoted by the same symbol), such that
\[
 \psi_n(t)  \xrightarrow{n \to \infty}  \psi_*(t) \text{ in $L^2(\R^d)$, for almost all $t\in[0,T]$.}
\]
In order to obtain weak convergence in the energy space, i.e. $\psi_n(t) \rightharpoonup \psi_*(t) \text{ in } \Sigma$, we 
fix $t\in[0,T]$ such that $\psi_n(t) \to \psi_*(t)$ in $L^2(\R^d)$. In view of~\eqref{eq:H1_bounds}, every subsequence of $\psi_n(t)$ has yet another subsequence 
such that $\psi_n(t)$ converges weakly in $\Sigma$ to some limit. On the other hand, this limit is necessarily given by $\psi_*(t)$, 
since $\psi_n(t) \to \psi_*(t)$ in $L^2(\R^d)$. Hence the whole sequence converges weakly in $\Sigma$ to $\psi_*(t)$. 
By lower-semicontinuity of the $\Sigma$--norm we can deduce $\|\psi_n(t)\|_{\Sigma}\le C$ and thus $\psi_* \in L^\infty(0,T;\Sigma)$. 

Finally, the announced convergence in $L^2(0,T;L^{2\sigma+2}(\R^d)$) is obtained by invoking the Gagliardo--Nirenberg inequality, i.e.
\begin{equation}
\label{eq:GN}
 \| \xi \|_{L^r_x} \le C \| \xi \|_{L^2_x}^{1-\delta(r)} \| \nabla\xi \|_{L^2_x}^{\delta(r)},
\end{equation}
where $2\le r < \frac{2d}{d-2}$ and $\delta(r) = d (\frac{1}{2}-\frac{1}{r})$. This concludes the proof of Proposition~\ref{prop:min_limit}.
\end{proof}

\subsection{Minimizers as mild solutions}\label{sec:mild}

Next we prove that the limit $\psi_*$ obtained in the previous subsection is indeed a mild solution of \eqref{eq:schroed} with corresponding control $\alpha_*$. 
From the physical point of view, this is important since it implies continuous (in time) dependence of $\psi_*$ upon a given initial data $\psi_0$. To this end, one should also 
note that $H^1(0,T)\hookrightarrow C(0,T)$ (using Sobolev imbeddings), and hence the obtained optimal control parameter $\alpha_*(t)$ is indeed a continuous function on $[0,T]$.

\begin{proposition}
\label{prop:NLS_sol}
Let $(\psi_*,\alpha_*)\in \Upsilon(0,T)\times H^1(0,T)$ be the limit obtained in Proposition~\ref{prop:min_limit}. Then $\psi_*$ is a mild solution of \eqref{eq:schroed} with control $\alpha_*$ and 
\[\psi_*\in C([0,T];\Sigma)\cap C^1([0,T];\Sigma^*).\]
In particular, this implies that the convergence result \eqref{eq:loc_sigma_conv} holds for all $t\in[0,T]$.
\end{proposition}
\begin{proof}
First we note that, by construction, each $\psi_n$ satisfies
\[
 \psi_n(t) = S(t)\psi_0 - i \int_0^t S(t-s) \left( \lambda |\psi_n(s)|^{2\sigma}\psi_n(s) + \alpha_n(s) V \psi_n(s) \right)ds
\]
for all $t\in [0,T]$.
Here and in the following we shall suppress the $x$--dependence of $\psi$ for notational convenience.
In order to prove that $\psi_*$ is a mild solution corresponding to the control $\alpha_*$, we take the $L^2$--scalar product of the above equation with a test function 
$\chi\in C^\infty_0(\R^d)$. This yields
\begin{equation}\label{eq:mild_NLS} 
\begin{split}
 \langle \psi_n(t), \chi \rangle_{L^2_x}= & \ \langle S(t)\psi_0 , \chi \rangle_{L^2_x} - i \lambda \int_0^t \big\langle S(t-s)  |\psi_n(s)|^{2\sigma}\psi_n(s), \chi \big\rangle_{L^2_x} \, ds\\
&\  -i\int_0^t \big \langle S(t-s)  \alpha_n(s)   V \psi_n(s), \chi \big \rangle_{L^2_x} \;ds.
\end{split}
\end{equation}
In view of Proposition~\ref{prop:min_limit}, the term on the left hand side of this identity converges to the desired expression for almost all $t\in [0,T]$, i.e.
\[
\lim_{n\to \infty} \langle \psi_n(t), \chi \rangle_{L^2_x} = \langle \psi_*(t), \chi\rangle_{L^2_x}.
 \]
In order to proceed further, we note that for any $f\in \mathcal D'(\R^d)$ it holds that
\begin{equation}\label{eq:ident}
 \langle S(t-s) f(s), \chi \rangle_{L^2_x} = \langle f(s), S(s-t) \chi \rangle_{L^2_x},
\end{equation}
and we therefore define
\begin{equation}
\label{eq:chi}
\tilde{\chi}:[0,t]\times \R^d \rightarrow\C, \quad \chi \mapsto \tilde \chi(\cdot, x):=S(\cdot - t)\chi(x),
\end{equation}
for which we can prove the following regularity properties.
\begin{lemma}
 \label{lem:tilde_omega}
There exists a constant $C=C(T)>0$ such that for all $t\in [0,T]$ it holds that
\begin{align*}
\sup_{s\in [0,t]} \left ( \|x \tilde{\chi}(s)\|_{L^2_x} + \|\nabla \tilde{\chi}(s)\|_{L^2_x} \right ) & \le C(T) < +\infty,
\end{align*}
where the function $\tilde{\chi}$ is defined in~\eqref{eq:chi}. In particular, the function $\tilde{\chi}$ is bounded in $L^\infty(0,t;L^{2\sigma+2}(\R^d))$.
\end{lemma}
\begin{proof}[Proof of Lemma \ref{lem:tilde_omega}.]
The norm $\|\tilde{\chi}(s)\|_{L^2_x} = \|S(s-t)\chi\|_{L^2_x}$ is conserved since $S(t)$ is a unitary operator on $L^2(\R^d)$. Furthermore, it holds that
\[
 i\partial_t [\nabla, S(t)] = H [\nabla, S(t)] + [\nabla, H] S(t) = H [\nabla, S(t)] + \nabla U S(t),
\]
and hence 
\[
 [\nabla, S(t)] = -i \int_0^t S(t-s) \nabla U\ S(s) \;ds.
\]
We can thus estimate
\begin{equation}
\label{eq:nabla_chi}
\begin{split}
 \|\nabla \tilde{\chi}(s)\|_{L^2_x} &= \|\nabla S(t-s) \chi\|_{L^2_x}\\
 & \le \|S(t-s) \nabla \chi\|_{L^2_x} + \Big \|\int_0^{t-s} S(t-s-\tau)\nabla U \tilde{\chi}(\tau) \;d\tau \Big \|_{L^2_x}\\
 &\le \|\nabla\chi\|_{L^2_x} + C \int_0^{t-s} \|x\tilde{\chi}(\tau)\|_{L^2_x}\;d\tau,
\end{split}
\end{equation}
since $U$ is subquadratic, i.e.\ $|\nabla U(x)|\le C|x|$. Likewise, we deduce
\[
 [x, S(t)] = -i \int_0^t S(t-s) \nabla S(s) \;ds
\]
and hence
\begin{equation}
 \label{eq:x_chi}
 \|x\tilde{\chi}(s)\|_{L^2_x} \le \|x\chi\|_{L^2_x} + \int_0^{t-s} \|\nabla \tilde{\chi}(\tau)\|_{L^2_x}\;d\tau.
\end{equation}
Combining the estimates~\eqref{eq:nabla_chi} and~\eqref{eq:x_chi} and applying Gronwall's inequality yields
\[
 \|x \tilde{\chi}(s)\|_{L^2_x} + \|\nabla \tilde{\chi}(s)\|_{L^2_x} \le C \left(\|x \chi\|_{L^2_x} + \|\nabla \chi\|_{L^2_x}\right) < +\infty,
\]
where $C>0$. The bound in $L^\infty(0,t;L^{2\sigma+2}(\R^d))$ then 
follows from the uniform-in-time bound in $H^1(\R^d)$ and the Gagliardo--Nirenberg inequality~\eqref{eq:GN}.
\end{proof}

With the result of Lemma \ref{lem:tilde_omega} at hand, we consider the second term on the right hand side of \eqref{eq:mild_NLS}. Rewriting it using \eqref{eq:ident}, we 
estimate
\begin{align*}
& \left|\int_0^t \big \langle |\psi_n(s)|^{2\sigma}\psi_n(s) - |\psi_*(s)|^{2\sigma}\psi_*(s), \tilde{\chi}(s)\big \rangle_{L^2_x} \, ds \right| \\
& \le  \int_0^t \int_{\R^d} \left| |\psi_n(s,x)|^{2\sigma}\psi_n(s,x) - |\psi_*(s,x)|^{2\sigma}\psi_*(s,x) \right| |\tilde{\chi}(s,x)| \;dx \, ds\\
& \le \  C  \int_0^t \int_{\R^d} \left( |\psi_n(s,x)|^{2\sigma} + |\psi_*(s,x)|^{2\sigma} \right) |\psi_n(s,x)-\psi_*(s,x)|  |\tilde{\chi}(s,x)| \, dx \, ds.
 \end{align*}
By H\"older's inequality, it holds that
\begin{align*}
& \int_0^t \int_{\R^d} \left| |\psi_n(s,x)|^{2\sigma} + |\psi_*(s,x)|^{2\sigma} \right| |\psi_n(s,x)-\psi_*(s,x)|  |\tilde{\chi}(s)| \, dx \, ds\\
& \le \ \sqrt{T} \left( \|\psi_n\|_{L^\infty_t L^{2\sigma+2}_x}^{2\sigma} + \|\psi_*\|_{L^\infty_t L^{2\sigma+2}_x}^{2\sigma} \right) \|\psi_n-\psi_*\|_{L^2_t L^{2\sigma+2}_x}  \|\tilde{\chi}\|_{L^\infty_t L^{2\sigma+2}_x},
\end{align*}
where, in view of Lemma~\ref{lem:tilde_omega}, we have  $\|\tilde{\chi}\|_{L^\infty_t L^{2\sigma+2}_x} <+ \infty$. In addition, Proposition~\ref{prop:min_limit} implies that 
the factor inside the parentheses is bounded and that $$\lim_{n \to \infty} \|\psi_n-\psi_*\|_{L^2_t L^{2\sigma+2}_x} =0.$$ 
Thus, we have shown that the second term on the right hand side of \eqref{eq:mild_NLS} vanishes in the limit $n\to\infty$. 

It remains to treat the last term on the right hand side of \eqref{eq:mild_NLS}, rewritten via \eqref{eq:ident}. We first estimate 
\begin{align*}
 &\left|\int_0^t \big \langle \alpha_n (s) V \psi_n (s) - \alpha_*(s) V \psi_*(s), \tilde{\chi}(s) \big\rangle_{L^2_x} \, ds \, \right|\\
 &\le \int_0^t\int_{\R^d} |\alpha_n(s)| \, |V(x)| \, |\psi_n(s,x) - \psi_*(s,x)| \, |\tilde{\chi}(s,x)| \;dx\, ds\\
 & \quad + \int_0^t\int_{\R^d} |\alpha_n(s)-\alpha_*(s)|\, |V(x)|\, |\psi_*(s,x)| |\tilde{\chi}(s,x)| \;dx\, ds.
\end{align*}
Here, the last term on the right hand side can be bounded by
\begin{align*}
& \int_0^t\int_{\R^d} |\alpha_n(s)-\alpha_*(s)| |V(x)| |\psi_*(s,x)| |\tilde{\chi}(s,x)| \;dx\, ds \\
& \le \|\alpha_n - \alpha_* \|_{L^2_t} \|V\|_{L^\infty_x} \|\psi\|_{L^2_tL^2_x} \|\tilde{\chi}\|_{L^\infty_tL^2_x} \xrightarrow{n \to \infty} 0,
\end{align*}
in view of the convergence of $\alpha_n\to \alpha_*$ 
in $L^2(0,T)$. For the remaining term we use the fact that $V\in L^\infty(\R^d)$ and H\"older's inequality to obtain that
\begin{align*}
 &\int_0^t\int_{\R^d} |\alpha_n(s)| \, |V(x)| \, |\psi_n(s,x) - \psi_*(s,x)| |\tilde{\chi}(s,x)| \;dx \, ds  \\ 
  &\le \ \|\alpha_n\|_{L^2_t} \|V\|_{L^\infty_x} \|\psi_n - \psi_*\|_{L^2_t L^2_x} \|\tilde{\chi}\|_{L^\infty_t L^2_x}  \xrightarrow{n \to \infty} 0,
\end{align*} 
due to the results of Proposition~\ref{prop:min_limit} and Lemma~\ref{lem:tilde_omega}.

In summary this proves that $\psi_* \in \Upsilon(0,T)$ satisfies, for \emph{almost all} $t\in[0,T]$,
\[
 \psi_*(t) = S(t)\psi_0 - i \int_0^t S(t-s) \left( \lambda |\psi_*(s)|^{2\sigma}\psi_*(s) + \alpha_*(s) V \psi_*(s) \right)ds,
\]
 i.e.\ $\psi_*$ is a \emph{weak $\Sigma$--solution} in the terminology of \cite[Definition 3.1.1]{caze} (where the analogous notion of weak $H^1$--solutions is introduced). 
 In order to obtain that $\psi_*$ is indeed a mild solution we note that
\[
 \psi_* \in \Upsilon(0,T) \hookrightarrow C([0,T];L^2(\R^d)) \cap C([0,T];L^{2\sigma+2}(\R^d))
\]
by interpolation and the Gagliardo--Nirenberg inequality~\eqref{eq:GN}. Classical arguments based on Strichartz estimates then yield uniqueness of the weak $\Sigma$--solution $\psi_*$. 
Arguing as in the proof of \cite[Theorem 3.3.9]{caze}, we infer that $\psi_*$ is indeed a mild solution to~\eqref{eq:schroed}, satisfying $\psi_*\in C([0,T];\Sigma)\cap C^1([0,T];\Sigma^*)$.
\end{proof}

\subsection{Lower semicontinuity of objective functional}\label{sec:lsc_J}
In order to conclude that the pair $(\psi_*,\alpha_*)\in \Lambda(0,T)$ is indeed a minimizer of our optimal control problem, it remains to show 
lower semicontinuity of the functional $J(\psi, \alpha)$ with respect to the convergence results established in Proposition~\ref{prop:min_limit}. 
\begin{lemma}
\label{lem:lsc}
For the sequence constructed in Proposition~\ref{prop:min_limit}, it holds that
\[
 J_* = \liminf_{n\rightarrow\infty}J(\psi_n,\alpha_n) \ge J(\psi_*,\alpha_*).
\]
\end{lemma}
\begin{proof}
Since $A\in\mathcal{L}(\Sigma, L^2(\R^d))$ by assumption, the sequence $(A\psi_n(T))_{n\in\N}$ converges weakly to $A\psi(T)$ in $L^2(\R^d)$. In addition 
$\psi_n(T) \to \psi_*(T)$ in $L^2(\R^d)$ as $n\to\infty$ by Proposition~\ref{prop:NLS_sol}, and hence the estimate
\begin{align*}
 &\left| \langle \psi_n(T), A\psi_n(T) \rangle_{L^2_x} - \langle \psi_*(T), A \psi_*(T) \rangle_{L^2_x} \right|\\ &\quad\le \left| \langle \psi_n(T) - \psi_*(T), A\psi_n(T) \rangle_{L^2_x} \right| + \left| \langle \psi_*(T), A (\psi_n(T) - \psi_*(T)) \rangle_{L^2_x} \right|.
\end{align*}
yields convergence of the corresponding term in the objective functional~\eqref{eq:J2}.
Next, we consider the cost term involving $\gamma_1$. In view of \eqref{eq:weight}, we define 
\[
\omega_n(t):=\int_{\R^d} V(x)|\psi_n(t,x)|^2 dx, \qquad \omega_*(t) := \int_{\R^d} V(x)|\psi_*(t,x)|^2 dx,
\]
and estimate
\begin{equation}
\begin{split}
\label{eq:lsc}
 &\liminf_{n\rightarrow\infty} \int_0^T (\dot{\alpha}_n(t))^2 \omega^2_n(t) dt \ge \\
& \liminf_{n\rightarrow\infty} \int_0^T (\dot{\alpha}_n(t))^2 \omega_*^2(t) \;dt + \liminf_{n\rightarrow\infty} \int_0^T (\dot{\alpha}_n(t))^2 \left(\omega^2_n(t) - \omega_*^2(t) \right) dt.
\end{split}
\end{equation}
Note that $0\le \omega_n(t) \le \|V\|_{L^\infty_x}\|\psi_0\|_{L^2_x}^2$ independently of $n\in\N$ and $t\in[0,T]$ and that the same holds for $\omega_*(t)$. The first term on the right hand side of \eqref{eq:lsc} is convex in $\alpha_n$ and thus satisfies
\begin{equation}
\label{eq:lsc_alpha}
 \liminf_{n\rightarrow\infty}\int_0^T (\dot{\alpha}_n(t))^2 \omega_*^2(t) \;dt \ge \int_0^T (\dot{\alpha_*}(t))^2 \omega_*^2(t) \;dt, 
\end{equation}
since any convex and lower semicontinuous functional is weakly lower semicontinuous.
On the other hand, Proposition~\ref{prop:min_limit} implies
\begin{equation}
\label{eq:lsc_g}
 \liminf_{n\rightarrow\infty} \omega_n(t) \ge \omega_*(t) \ge 0 \quad \text{for all $t\in[0,T]$.}
\end{equation}
Thus, using \eqref{eq:lsc_alpha} and~\eqref{eq:lsc_g} together with Fatou's Lemma yields
\begin{align*}
 &\liminf_{n\rightarrow\infty} \int_0^T (\dot{\alpha}_n(t))^2 \omega^2_n(t) \;dt \\
 &\ge  \int_0^T (\dot{\alpha_*}(t))^2 \omega_*^2(t) \;dt  + \int_0^T \liminf_{n\rightarrow\infty} (\dot{\alpha}_n(t))^2 \liminf_{n\rightarrow\infty} \left( \omega^2_n(t) - \omega_*^2(t) \right) dt  \\
 & \ge \int_0^T (\dot{\alpha_*}(t))^2 \omega_*^2(t) \;dt.
\end{align*}
Finally the cost term involving $\gamma_2$ is lower semicontinuous by convexity and weak convergence of $\alpha_n$ in $H^1(0,T)$.
\end{proof}
In summary, we have shown that $J_*=\liminf_{n\rightarrow\infty}J(\psi_n,\alpha_n) \ge J(\psi_*,\alpha_*)$ and thus indeed $J_* = J(\psi_*,\alpha_*)$. In other words, $(\psi_*, \alpha_*)\in \Lambda(0,T) $ 
solves the optimization problem.

\begin{remark} 
Note that the bound on $x\psi_n (t, \cdot)$ in $L^2(\R^d)$, obtained in Proposition \ref{prop:min_limit}, is indeed crucial for proving the weak lower-semicontinuity of $J(\psi, \alpha)$. 
Without such a bound on the second moment, we would only have $$\psi_n(t) \xrightarrow{n \to \infty} \psi (t) \quad \text{in $L^2_{\rm loc}(\R^d)$,}$$ 
due to the lack of compactness of $H^1(\R^d)  \hookrightarrow L^2(\R^d)$. 
In this case, the lower semi-continuity of the term $\langle \psi(T), A\psi(T) \rangle_{L^2_x}$ is not guaranteed. 
A possible way to circumvent this problem would be to assume that $A$ is {\it positive definite}, which, however, is not true for general observables of the form $A= A'-a$, with $a\in \R$. 
A second possibility would be to assume that $A$ is {\it localizing}, i.e. for all $\psi \in H^1(\R^d)$: $\text{supp}_{x\in \R^d}(A\psi(x)) \subseteq B(R)$, for some $R<+\infty$.  
\end{remark}

\section{Derivation and analysis of the adjoint equation}\label{sec:adjoint}

In order to give a characterization of a minimizer $(\psi_*, \alpha_*) \in \Lambda(0,T)$, we need to derive the 
first order optimality conditions for our optimal control problem~\eqref{eq:minprob}. 
For this purpose, we shall first formally compute the derivative of the objective functional $J(\psi, \alpha)$ in the next subsection and consequently analyze the resulting adjoint problem. 
A rigorous justification for the derivative will be given in Section \ref{sec:derivative}. 

\subsection{Identification of the derivative of $J(\psi, \alpha)$} \label{sec:ident}
The mild solution of the nonlinear Schr\"odinger equation \eqref{eq:schroed}, corresponding to the control $\alpha\in H^1(0,T)$, induces a map
\[
\psi : H^1(0,T) \rightarrow \Upsilon(0,T): \quad \alpha \mapsto \psi(\alpha).
\]
Using this map we introduce the unconstrained or \emph{reduced functional}
\[
 \mathcal J : H^1(0,T) \rightarrow \R, \quad \alpha \mapsto \mathcal J(\alpha) := J(\psi(\alpha), \alpha).
\]
For the characterization of critical points, we need to compute the derivative of $\mathcal{J}$. For this calculation let 
$\delta_\alpha \in H^1(0,T)$ with $\delta_\alpha(0)=0$ be a feasible control perturbation. (Recall that $H^1(0,T) \hookrightarrow C(0,T)$ and hence it makes sense 
to evaluate $\delta_\alpha(t)$ at $t=0$.) Then the chain rule yields
\begin{equation}
 \label{eq:J_prime2}
 \begin{split}
 \langle \mathcal J' (\alpha) , \delta_\alpha \rangle_{H^{-1}_t, H^1_t}=  &\ \langle \partial_\psi J(\psi(\alpha), \alpha), \psi'(\alpha) \delta_\alpha \rangle_{\Upsilon^\ast, \Upsilon}   \\
 & \ + \langle \partial_\alpha J(\psi(\alpha),\alpha), \delta_\alpha \rangle_{H^{-1}_t, H_t^1}
\end{split}
\end{equation}
where $\Upsilon^*$ denotes the dual space of $\Upsilon\equiv \Upsilon(0,T)$ for any given $T>0$. 
The main difficulty lies in computing $\psi'(\alpha)$ since $\psi$ is given only implicitly through the nonlinear Schr\"odinger equation~\eqref{eq:schroed}. 

In the following, we shall write the (nonlinear) partial differential equation \eqref{eq:schroed} in a more abstract form, i.e.
\begin{equation} \label{eq:equation}
 P(\psi,\alpha) :=  i \partial_t \psi - H \psi- \alpha(t) V(x) \psi - \lambda |\psi|^{2\sigma} \psi =0,
\end{equation}
where $H = -\frac{1}{2}\Delta +U(x)$ denotes the linear, uncontrolled Hamiltonian operator. Setting $\psi = \psi(\alpha)$ and differentiating with respect to $\alpha$ formally yields
\[
\frac{d}{d\alpha} P(\psi(\alpha),\alpha)= \partial_\psi P(\psi(\alpha),\alpha) \psi'(\alpha) + \partial_\alpha P(\psi(\alpha),\alpha) = 0.
\]
Next, assuming that $\partial_\psi P$ is invertible, we solve for $\psi'(\alpha)$ via
\[
 \psi'(\alpha) = - \partial_\psi P(\psi,\alpha)^{-1} \partial_\alpha P(\psi(\alpha),\alpha).
\]
Thus it holds that
\begin{align*}
&  \left\langle \partial_\psi J(\psi(\alpha),\alpha), \psi'(\alpha) \delta_\alpha \right\rangle_{\Upsilon^\ast,\Upsilon}\\ 
&\ = \left\langle - \partial_\psi J(\psi(\alpha), \alpha), \partial_\psi P(\psi(\alpha),\alpha)^{-1} \partial_\alpha P(\psi(\alpha),\alpha) \delta_\alpha \right\rangle_{\Upsilon^\ast,\Upsilon},
\end{align*}
which can be rewritten as
\begin{equation}
\label{eq:J_psi_prime}
\begin{split}
& \ \left\langle \partial_\psi J(\psi(\alpha),\alpha), \psi'(\alpha) \delta_\alpha \right\rangle_{\Upsilon^\ast,\Upsilon}\\
&\ = \left\langle - \partial_\alpha P(\psi(\alpha),\alpha)^\ast \partial_\psi P(\psi(\alpha),\alpha)^{-\ast} \partial_\psi J(\psi(\alpha), \alpha), \delta_\alpha \right\rangle_{H_t^{-1},H^1_t}.
\end{split}
\end{equation}
Here we abbreviate $$\partial_\psi P(\psi(\alpha),\alpha)^{-\ast}:=(\partial_\psi P(\psi(\alpha),\alpha)^\ast)^{-1} = (\partial_\psi P(\psi(\alpha),\alpha)^{-1})^\ast.$$ 
Substituting \eqref{eq:J_psi_prime} into equation~\eqref{eq:J_prime2}, we see that critical points of~\eqref{eq:minprob} satisfy
\begin{equation}
 \label{eq:stat_point}
\begin{split}
 0=& \ \langle \mathcal{J}'(\alpha), \delta_\alpha \rangle_{H^{-1}_t, H^1_t} = \langle \partial_\alpha J(\psi(\alpha),\alpha), \delta_\alpha \rangle_{H^{-1}_t, H^1_t} +\\ 
& \ \langle -\partial_\alpha P(\psi(\alpha),\alpha)^* \partial_\psi P(\psi(\alpha),\alpha)^{-\ast} \partial_\psi J(\psi(\alpha), \alpha), \delta_\alpha \rangle_{H^{-1}_t,H^1_t}
\end{split}
\end{equation}
for all $\delta_\alpha \in H^1(0,T)$ such that $\delta_\alpha(0) = 0$. In order to obtain \eqref{eq:stat_point} in a more explicit form, we (formally) compute the derivative
\begin{align}\label{eq:linNLS}
 \partial_\psi P(\psi,\alpha) \xi = i \partial_t \xi -H \xi - \alpha(t) V (x) \xi - \lambda (\sigma+1) |\psi|^{2\sigma} \xi - \lambda \sigma |\psi|^{2\sigma-2} \psi^2 \conj{\xi},
 \end{align}
acting on $\xi \in L^2(\R^d) \subset \Sigma^*$. Analogously, we find
\begin{align*}
 \partial_\alpha P(\psi,\alpha) = - V(x) \psi.
\end{align*}
Next, we define 
\begin{equation}\label{defphi}
 \varphi := \partial_\psi P(\psi(\alpha),\alpha)^{-\ast} \partial_\psi J(\psi(\alpha), \alpha),
\end{equation}
which, in view of \eqref{eq:stat_point}, allows us to express $\mathcal J'(\alpha)\in(H^1(0,T))^*$ in the following form:
\begin{equation}
\label{eq:grad_weak}
 \mathcal{J}'(\alpha) = \partial_\alpha J(\psi(\alpha),\alpha) - \partial_\alpha P(\psi(\alpha),\alpha)^* \varphi.
\end{equation}
We consequently obtain $ \mathcal{J}'(\alpha)$ by explicitly calculating the right hand side of this equation (given in \eqref{eq:J_prime} below), provided we can determine $\varphi$. 

In order to perform this calculation, we recall that the duality pairing between $\xi \in L^2(\R^d) \subset \Sigma^*$ and $\psi \in \Sigma$ can be expressed 
by the inner product defined in~\eqref{eq:inner_prod}. Thus, \eqref{defphi} implies
\begin{equation} \label{eq:weak_adjoint}
 \langle \varphi, \partial_\psi P(\psi(\alpha),\alpha) \delta_\psi \rangle_{L^2_tL^2_x} = \langle \partial_\psi J(\psi(\alpha),\alpha), \delta_\psi \rangle_{L^2_tL^2_x},
\end{equation}
for all test functions $\delta_\psi\in \Upsilon(0,T)$ such that $\delta_\psi(0) = 0$. This is the correct ``tangent space" for $\psi$ in view of the Cauchy data
\[
\psi(0) + \delta_\psi(0) = \psi_0 \text{ and } \psi(0) = \psi_0.
\]
By virtue of the symmetry of the {\it linearized operator} $\partial_\psi P(\psi(\alpha),\alpha)$, 
equation~\eqref{eq:weak_adjoint} corresponds to the weak formulation of the following \emph{adjoint equation}: 
\begin{equation}\label{eq:adjoint1}
\left \{
\begin{aligned}
& i \partial_t \varphi - H \varphi - \alpha(t) V (x)\varphi - \lambda (\sigma+1) |\psi|^{2\sigma} \varphi - \lambda \sigma |\psi|^{2\sigma-2} \psi^2 \conj {\varphi} = \frac{\delta J(\psi,\alpha)}{\delta\psi(t)},\\
& \text{for all $t\in[0,T]$ and with data: } \varphi(T) = i \frac{\delta J(\psi,\alpha)}{\delta\psi(T)}.
\end{aligned}
\right.
\end{equation}
Here, $ \frac{\delta J(\psi,\alpha)}{\delta\psi(t)}$ denotes the first variation of $J(\psi, \alpha)$ 
with respect to the value of $\psi(t)\in H^1(\R^d)$, where $\psi$ is the solution of \eqref{eq:schroed} with control $\alpha$. 
Likewise, $ \frac{\delta J(\psi,\alpha)}{\delta\psi(T)}$ denotes the first variation with respect to solutions of \eqref{eq:schroed} evaluated at the final time $t=T$.
Explicitly, these derivatives are given by
\begin{equation}\label{eq:J_psi1}
\begin{split}
 \frac{\delta J(\psi,\alpha)}{\delta\psi(t)} = & \ 4 (\dot{\alpha}(t))^2  \left( \int_{\R^d} V (x)|\psi(t,x)|^2 dx \right) V(x)\psi (t, x)\\
\equiv  & \  4 (\dot{\alpha}(t))^2 \omega(t) V(x)\psi (t, x),
\end{split}
 \end{equation}
in view of the definition \eqref{eq:weight}, and 
\begin{align}
\label{eq:J_psi2}
 \frac{\delta J(\psi,\alpha)}{\delta\psi(T)} &= 4 \langle \psi(T, \cdot), A \psi(T, \cdot) \rangle_{L^2_x} A \psi(T, x).
\end{align}
The system \eqref{eq:adjoint1} consequently defines a Cauchy problem for $\varphi$ with data given at $t=T$, the final time. 
Thus, one needs to solve \eqref{eq:adjoint1} backwards in time, a common feature of adjoint systems for time-dependent phenomena.
\begin{remark}\label{rem:lag}
In fact, $\varphi$ can also be seen as a \emph{Lagrange multiplier} within the Lagrangian formulation of the optimal control problem. In oder to see this, one 
defines the Lagrangian
\[
L(\psi,\alpha,\varphi) = J(\psi,\alpha) - \langle \varphi, P(\psi,\alpha) \rangle_{L^2_t L^2_x},
\]
where $P(\psi, \alpha)$ is the nonlinear Schr\"odinger equation given in \eqref{eq:equation}.
Formally, the Euler--Lagrange equations associated to $L(\psi,\alpha,p)$ yield~\eqref{eq:grad_weak} and~\eqref{eq:adjoint1}. In Section~\ref{sec:num} we shall use the Lagrangian formulation 
to formally compute the Hessian of the reduced objective functional $\mathcal{J}(\alpha)$.\end{remark}

In the next subsection, we shall set up an existence theory for \eqref{eq:adjoint1}, which in turn will be used to rigorously justify the above derivation in Section~\ref{sec:derivative} below.

\subsection{Existence of solutions to the adjoint equation}\label{ssec:adjoint}

In order to obtain existence of solutions to \eqref{eq:adjoint1}, we need sufficiently high regularity of $\psi$, the solution of the Gross--Pitaevskii equation \eqref{eq:schroed}. 
For this purpose, for every $m\in \N$ we define
\begin{align*}
 \Sigma^m :=& \left\{ \psi \in L^2(\R^d)\ : \ x^{j} \nabla^{k} \psi \in L^2(\R^d) \text{ for all multi-indices $j$ and $k$ with }\right.\\&\qquad\qquad \left. |j| + |k| \le m  \right\},
\end{align*}
equipped with the norm (note that $\Sigma^1 \equiv \Sigma$): 
\[
 \| \psi \|_{\Sigma^m} := \sum_{|j| + |k| \le m} \left\| x^{j} \nabla^{k} \psi \right\|_{L^2_x}.
\]
\begin{remark} If the external potential $U(x)$ were in $L^\infty (\R^d)$, it would be enough to work in the space $H^m(\R^d)$ instead of $\Sigma^m$. 
In the presence of an external subquadratic potential, however, we also require control of higher moments of the wave function $\psi$ with respect to $x$.
\end{remark}

\begin{lemma} \label{lem:lin_Sigma^m} Let $S(t)$ be given by \eqref{eq:schg} with $U\in C^\infty(\R^d;\R)$ and subquadratic. Then, there exists a constant $c > 0$ such that
\[
 \|S(t) \psi_0 \|_{\Sigma^m} \le e^{ct }\|\psi_0\|_{\Sigma^m},
\]
for all $t\in [0,\infty)$ and $\psi_0 \in \Sigma^m$.
\end{lemma}
The proof of this lemma can be deduced by differentiating the $\Sigma^m$--norm with respect to time and applying Gronwall's inequality. It consequently implies 
the following regularity result for solutions to \eqref{eq:schroed}:
\begin{lemma}
\label{lem:Hm_solution}
Let $\lambda \ge 0$, $\sigma\in \N$ with $\sigma<2/(d-2)$, and $U \in C^\infty(\R^d)$ be subquadratic. For $m>d/2$, let $\psi_0 \in \Sigma^m$, and $V\in W^{m, \infty}(\R^d)$. 
Then the mild solution of \eqref{eq:schroed} satisfies $\psi\in L^\infty(0,T;\Sigma^m)$.
\end{lemma}
In view of Lemma \ref{lem:lin_Sigma^m}, this result can be proved by following the same arguments as in \cite[Theorem 5.5.1]{caze}. 
Having obtained $\psi \in L^\infty(0,T;\Sigma^m)$, we infer $\psi \in L^\infty((0,T)\times\R^d)$ by the Sobolev embedding $H^m(\R^d)\hookrightarrow L^\infty(\R^d)$ 
whenever $m>d/2$. Thus, all the $\psi$--dependent coefficients appearing in adjoint equation \eqref{eq:adjoint1} are indeed in $L^\infty$.
\begin{remark}
Note that Lemma \ref{lem:Hm_solution} requires us to impose $\sigma\in\N$, which together with the condition $\sigma<2/(d-2)$ necessarily implies $d\le 3$. 
The reason is that for general $\sigma>0$ (not necessarily an integer) the nonlinearity $| \psi |^{2\sigma}\psi$ is not locally Lipschitz in $\Sigma^m$ (cf. Lemma \ref{lem:tech_nonlin}) and 
the life-span of solution $\psi(t,\cdot) \in \Sigma^m$ is in general not known, see \cite{caze} for more details. \end{remark}

From now on, we shall always assume that $V\in W^{m, \infty}(\R^d)$ for $m>d/2$ and $U \in C^\infty(\R^d)$ subquadratic.
With the above regularity result at hand, classical semigroup theory \cite{pazy} allows us to construct a solution to the adjoint problem.

\begin{proposition}\label{prop:soladjoint} Let $\lambda \ge 0$, $\sigma\in \N$ with $\sigma<2/(d-2)$, and $U \in C^\infty(\R^d)$ be subquadratic. 
For $m>d/2$, let $\psi_0 \in \Sigma^m$, $V\in W^{m, \infty}(\R^d)$. Then, \eqref{eq:adjoint1} admits a unique mild solution
\[
 \varphi\in C([0,T];L^2(\R^d)).
\]
\end{proposition}
\begin{proof} First, we study the homogenous equation $ \partial_\psi P(\psi(\alpha),\alpha)\xi=0$, associated to \eqref{eq:adjoint1}. It can be written as
 \[
 \partial_t \xi =  -iH\xi + B(t)\xi,
 \] 
where
 \begin{align*}
 B(t)\xi :=  -i\left(\lambda (\sigma+1) |\psi|^{2\sigma}\xi  + \lambda \sigma |\psi|^{2\sigma-2} \psi^2 \conj {\xi} + \alpha(t) V (x)\xi\right).
\end{align*}
The operator $-iH: \Sigma^2\rightarrow L^2(\R^d)$ is simply the generator of the Schr\"odinger group $S(t)=e^{-iHt}$. On the other hand, for any $t\in[0,T]$, 
$B(t)$ is a linear operator on the real vector space $L^2(\R^d)$, equipped with the inner product \eqref{eq:inner_prod} (the same would not be true if we would consider $L^2(\R^d)$ as a complex 
vector space). In addition, $B(t)^*=B(t)$ is symmetric with respect to this inner product and the same is true for $iB(t)$. 
Since $V\in W^{m,\infty}(\R^d)$, $\alpha\in L^\infty(0,T)$ by assumption and $\psi\in L^\infty((0,T)\times\R^d)$ in view of Lemma \ref{lem:Hm_solution}, we infer $B \in L^\infty(0,T; \mathcal{L}(L^2(\R^d))).$ The operator 
$B(t)$ may therefore be considered as a (time-dependent) perturbation of the generator $-iH$.

Following the construction given in Proposition~1.2, Chapter 3 of \cite{pazy}, we obtain the existence of a propagator $F(t,s)$, i.e. 
a family of bounded operators
\[
 \{ F(t,s): L^2(\R^d) \rightarrow L^2(\R^d)\}_{s,t\in[0,T]}
\]
which are strongly continuous in time and satisfy $ F(t,s)=F(t,r)F(r,s)$. 
This propagator $F(t,s)$ is implicitly given by
\[
 F(t,s) = e^{-iH (t-s)} + \int_s^t e^{-iH (t-\tau)}B(\tau) F(\tau,s)\;d\tau
\]
and solves the homogeneous linearized equation in the sense that
\[
 \frac{d}{dt}F(t,s)\xi = \pare{-iH  +B(t)}F(t,s)\xi
\]
weakly in $(\Sigma^2)^*$ for every $\xi \in L^2(\R^d)$ and almost every $t\in[0,T]$. Clearly, it provides a unique mild solution $\xi (t) = F(t,s) \varphi(s)$ of the homogenous equation. 
Duhamel's formula applied to the adjoint problem~\eqref{eq:adjoint1} consequently yields
\begin{equation}
\label{eq:lin_sol}
 \varphi(t) = iF(t,T)\frac{\delta J(\psi,\alpha)}{\delta\psi(T)} + i \int_t^T F(t,s) \frac{\delta J(\psi,\alpha)}{\delta\psi(s)} \;ds.
\end{equation}
Under our assumptions on $\psi$ and $A$ we have that 
\[
 \frac{\delta J(\psi,\alpha)}{\delta\psi(t)}\in L^1(0,T;L^2(\R^d)), \quad  \frac{\delta J(\psi,\alpha)}{\delta\psi(T)}  \in L^2(\R^d),
\]
which in view of Duhamel's formula~\eqref{eq:lin_sol} implies the existence of a mild solution $\varphi \in C([0,T]; L^2(\R^d))$. Uniqueness follows from linearity and the uniqueness of the homogeneous equation.
\end{proof}

\section{Rigorous characterization of critical points}\label{sec:derivative}

A classical approach for making the derivation of the adjoint system rigorous is based on the implicit function theorem. The latter is used to show that $\partial_\psi P(\psi(\alpha), \alpha)$ is indeed invertible, but it 
requires the identification of a linear function space $X$ such that
\[
 P : \Upsilon(0,T) \times H^1(0,T) \rightarrow X \ ; (\psi, \alpha) \mapsto P(\psi, \alpha) ,
\]
and 
\[
 \partial_\psi P(\psi,\alpha)^{-1} :X \rightarrow \Upsilon(0,T).
\]
In other words, we require the solution of \eqref{eq:adjoint1} with a right hand side in $X$ to be in $\Upsilon(0,T)$. 
It seems, however, that the linearized operator $\partial_\psi P(\psi,\alpha)^{-1}$ is not sufficiently regularizing to allow for an easy identification of $X$. 
Therefore we shall not invoke the implicit function theorem but rather calculate the G\^{a}teaux-derivative $\mathcal J'(\alpha)$ directly. (We do not prove Fr\'{e}chet-differentiability; see Remark~\ref{rem:gateaux} below.) 
To this end, we shall first show that the solution $\psi=\psi(\alpha)$ to \eqref{eq:schroed} depends Lipschitz-continuously on the control parameter $\alpha$. 
This will henceforth be used to estimate the error terms appearing in the derivative of $\mathcal{J}(\alpha)$. 

\subsection{Lipschitz continuity with respect to the control}
\label{sec:lipschitz}

As a first step towards full Lipschitz continuity, we prove local-in-time Lipschitz continuity of $\psi=\psi(\alpha)$ with respect to the control parameter $\alpha$. 
\begin{proposition}
\label{prop:loc_Lipschitz}
 Let $\lambda\ge0$, $\sigma\in \N$ with $\sigma<2/(d-2)$, and $U\in C^\infty(\R^d)$ be subquadratic. For $m>d/2$, let $V\in W^{m,\infty}(\R^d)$ and 
 $\tilde{\psi}, \psi \in L^\infty(0,T; \Sigma^m)$ be two mild solutions to \eqref{eq:schroed}, corresponding to initial data 
 $\tilde{\psi}_0, \psi_0\in \Sigma^m$ and control parameters $\tilde{\alpha}, \alpha \in H^1(0,T)$, respectively. 
 Assume that
 \[\|\tilde{\alpha}\|_{H^1_t}, \|\alpha\|_{H^1_t}, \|\tilde{\psi}(t, \cdot)\|_{\Sigma^m}, \|\psi(t, \cdot)\|_{\Sigma^m}\le M\]
 for some given $M\ge 0$.
 Then there exist $\tau=\tau(M)>0$ and a constant $C=C(M) < +\infty$, such that 
 \begin{equation}
\label{eq:loc_Lipschitz}
  \| \tilde{\psi} - \psi \|_{L^\infty(I_t; \Sigma^m)} \le C \left( \| \tilde{\psi}(t) - \psi(t) \|_{\Sigma^m} + \| \tilde{\alpha} - \alpha \|_{H^1_t} \right),
 \end{equation}
where $I_t := [t,t+\tau]\cap [0,T]$. In particular, the mapping $\alpha \mapsto \psi(\alpha)\in \Upsilon (0,T)$ is continuous with respect to $\alpha\in H^1(0,T)$.
\end{proposition}
\begin{proof} To simplify notation, let us assume $t+\tau \le T$. By construction, there exists a $\tau>0$ depending only on $M$, such that $\psi|_{I_t}$ is a fixed point of the mapping
 \[
  \psi \mapsto S(\,\cdot\,)\psi_0 - i \int_t^{\, \cdot} S(\cdot-s) \left( \lambda|\psi(s)|^{2\sigma}\psi(s) + \alpha(s) V \psi(s) \right) \;ds,
 \]
 which maps the set $$Y=\{ \psi \in L^\infty(I_t;\Sigma^m)  :  \| \psi \|_{L^\infty(I_t;\Sigma^m)} \le 2 M \}$$ into itself. 
 Of course, the same holds true for $\tilde{\psi}$ and $\tilde{\alpha}$ in place of $\psi$ and $\alpha$, respectively. In particular, the embedding $\Sigma^m(R^d)\hookrightarrow L^\infty(\R^d)$, $m>d/2$, yields
\[
 \| \psi \|_{L^\infty(I_t\times\R^d)} \le 2 C M.
\]
To proceed further, we recall the following result, which can be proved along the lines of \cite[Lemma~4.\,10.\,2]{caze}. 
\begin{lemma}
 \label{lem:tech_nonlin}
Let $M>0$, $\sigma \in \N$, and $m>d/2$. Then there exists a constant $C(M)>0$, such that for all $\psi,\tilde{\psi}\in \Sigma^m$ satisfying $\|\psi\|_{L^\infty_x}, \|\tilde{\psi}\|_{L^\infty_x}\le M$, it holds that
\[
 \big \||\psi|^{2\sigma} \psi - |\tilde{\psi}|^{2\sigma} \tilde{\psi}\big\|_{\Sigma^m} \le C(M) \|\psi-\tilde{\psi}\|_{\Sigma^m}.
\]
In other words, $\psi \mapsto |\psi|^{2\sigma} \psi$ is locally Lipschitz in $\Sigma^m$.
\end{lemma}
Subtracting the two fixed point expressions for $\tilde{\psi}$ and $\psi$ gives
\begin{align*}
&\tilde{\psi}(s) - \psi(s) =  \ S(s-t)(\tilde{\psi}(t)-\psi(t))\\ & \qquad- i\int_0^{s-t} S(s-r)\Big( \lambda(|\tilde{\psi}|^{2\sigma}\tilde{\psi}-|\psi|^{2\sigma}\psi) 
   + V(x) (\tilde{\alpha} \, \tilde{\psi}- \alpha \, \psi) \Big)(\tau) \;d\tau
\end{align*}
for all $s\in[t,t+\tau]$.
Taking the $L^\infty(I_t;\Sigma^m)$-norm and recalling Lemma \ref{lem:lin_Sigma^m}, together with $\|\psi(s)\|_{\Sigma^m},\|\tilde{\psi}(s)\|_{\Sigma^m}\le 2M$, for 
$s\le t+\tau$, yields
 \begin{align*}
  \|\tilde{\psi} - \psi\|_{L^\infty(I_t;\Sigma^m)} & \le C \|\tilde{\psi}(t) - \psi(t)\|_{\Sigma^m} +
   2 M \|\tilde{\alpha}-\alpha\|_{H^1_t} \|V\|_{W^{m,\infty}_x}\\& + C \tau \left( C(2M) + \|\tilde{\alpha}\|_{H^1_t} \|V\|_{W^{m,\infty}_x} \right) \|\tilde{\psi} - \psi\|_{L^\infty(t,t+\tau;\Sigma^m)},
 \end{align*}
 where $C(2M)$ is the constant appearing in Lemma~\ref{lem:tech_nonlin} with $2M$ replacing $M$. Since $\|\tilde{\alpha}\|_{H^1_t}\le M$, 
 the estimate~\eqref{eq:loc_Lipschitz} follows from possibly choosing $\tau$ even smaller.\\
 Finally, we show the continuity of the map $H^1(0,T) \rightarrow \Upsilon(0,T), \alpha \mapsto \psi(\alpha)$. Set
\[
 t_* := \inf \big\{ 0 \le t \le T : \limsup_{\tilde{\alpha}\to\alpha} \|\tilde{\psi} - \psi\|_{L^\infty(0,t;\Sigma^m)} > 0 \big\},
\]
with the convention $\inf\emptyset := +\infty$. We have to show that $t_* = +\infty$. 
Assuming $t_*\le T <+\infty$, fix $M'\ge M$ such that $\|\psi\|_{L^\infty(0,T;\Sigma^m)} \le M'$, let $\tau'=\tau(M'+1) > 0$ be chosen as above, with $M'+1$ replacing $M$. Furthermore let $\Delta t = \tau'/2$. The definition of of $t_*$ yields
\[
 \limsup_{\tilde{\alpha} \to \alpha} \|\tilde{\psi} - \psi\|_{L^\infty(0,t_* - \Delta t;\Sigma^m)} = 0.
\]
In particular, it holds that $\|\tilde{\psi}\|_{L^\infty(0,t_* - \Delta t;\Sigma^m)} \le M'+1$ for all $(\tilde{\alpha} - \alpha)$ small enough. 
But now we see that the Lipschitz continuity~\eqref{eq:loc_Lipschitz} is satisfied by $\tilde\psi$ and $\psi$ and such controls $\tilde{\alpha}$, $\alpha$ on the interval $[t_* - \Delta t, t_* - \Delta t + \tau']$. Hence
\[
 \limsup_{\tilde{\alpha} \to \alpha} \|\tilde{\psi} - \psi\|_{L^\infty(0,t_* - \Delta t + \tau';\Sigma^m)} = 0,
\]
a contradiction to the definition of $t_*$. Hence we must have $t_* = \infty$, and continuity holds.
\end{proof}
As a direct consequence of this continuity result, we obtain uniform boundedness of the solution $\psi(\alpha)$ on compact sets in $\alpha\in H^1(0,T)$. 
Of course, bounded sets in $H^1(0,T)$ are in general not compact and thus we have to restrict ourselves to finite-dimensional subsets.
\begin{corollary}
 \label{cor:psi_bound}
 Under the assumptions of Proposition~\ref{prop:loc_Lipschitz}, let $\delta_\alpha \in H^1(0,T)$ with $\delta_\alpha(0) = 0$ be a direction of change for $\alpha$ 
 and let $\psi(\alpha+\varepsilon \delta_\alpha)$ be the solution to \eqref{eq:schroed} with control $\alpha+\varepsilon \delta_\alpha$ and initial data $\psi_0\in \Sigma^m$, $m>d/2$. 
 Then there exists $M<\infty$ such that
\[
 \| \psi(\alpha+\varepsilon \delta_\alpha) \|_{L^\infty(0,T;\Sigma^m)} \le M, \quad \forall \, \varepsilon \in [-1,1].
\]
\end{corollary}
\begin{remark}
\label{rem:gateaux}
This bound on finite dimensional subsets of $H^1(0,T)$ is the reason why we can only prove G\^{a}teaux-differentiability. 
If we had a bound on $\psi(\alpha)$ in the $\Sigma^m$--norm which was uniform in $t\le T$ and $\|\alpha\|_{H^1_t}\le M$, we could prove Fr\'{e}chet-differentiability. 
For our further analysis, however, this will not be of any consequence.
\end{remark}
Now we are ready to prove Lipschitz-continuity of the solution $\psi(\alpha)$ with respect to the control parameter $\alpha \in H^1(0,T)$ on the whole control interval $[0,T]$.
\begin{proposition}
 \label{prop:lipschitz_in_alpha}
Let $\lambda\ge0$, $\sigma\in \N$ with $\sigma<2/(d-2)$, and $U\in C^\infty(\R^d)$ be subquadratic. 
For $m > d/2$, let $V\in W^{m,\infty}(\R^d)$, $\psi_0 \in \Sigma^m$, $\alpha\in H^1(0,T)$, and $\psi\equiv \psi(\alpha)\in L^\infty(0,T; \Sigma^m)$ be the solution to \eqref{eq:schroed}. Set $\tilde{\psi} \equiv \psi(\tilde{\alpha})$ where for any $\varepsilon \in [-1,1]$, we let
 $\tilde{\alpha}:=\alpha+\varepsilon\delta_\alpha$ with $\delta_\alpha \in H^1(0,T)$ such that $\delta_\alpha(0)=0$. Then there exists a constant $C>0$, such that
\begin{equation}
 \label{eq:lipschitz_in_alpha}
 \|\tilde{\psi} - \psi\|_{L^\infty(0,T;\Sigma^m)} \le C \|\tilde{\alpha} - \alpha\|_{H^1(0,T)} = C |\varepsilon| \|\delta_\alpha\|_{H^1(0,T)}.
\end{equation}
In other words, the solution to \eqref{eq:schroed} depends Lipschitz-continuously on the control $\alpha$ for each fixed direction $\delta_\alpha$.
\end{proposition}
\begin{proof}
Since Corollary~\ref{cor:psi_bound} provides a uniform (in $\varepsilon$) bound on $\|\tilde{\psi}\|_{L^\infty(0,T;\Sigma^m)}$, 
the quantity $\tau$ in the local Lipschitz estimate~\eqref{eq:loc_Lipschitz} is now independent of $\epsilon$ and $t$ and the estimate indeed holds on every interval $[t, t + \tau]$, i.e.\
\[
 \| \tilde{\psi}- \psi \|_{L^\infty(t,t+\tau;\Sigma^m)} \le C \left( \| \tilde{\psi}(t) - \psi(t) \|_{\Sigma^m} + \| \tilde{\alpha} - \alpha \|_{H^1(t,t+\tau)} \right).
\]
Since both solutions $\tilde{\psi}$ and $\psi$ coincide at $t=0$, finite summation of this estimate over intervals $[n\tau,(n+1)\tau]$ yields \eqref{eq:lipschitz_in_alpha}.
\end{proof}

\subsection{Proof of differentiability and characterization of critical points}\label{sec:differ}

We are now in a position to state the second main result of this work.

\begin{theorem} \label{thm:J_prime} 
Let $\lambda\ge0$, $\sigma \in \N$ with $\sigma<2/(d-2)$, and $U\in C^\infty(\R^d)$ be subquadratic. In addition, let $\psi_0\in \Sigma^m$, $V\in W^{m,\infty}(\R^d)$ for some $m\in \N$, $m \ge 2$, and $\alpha\in H^1(0,T)$.\\ 
Then the solution of \eqref{eq:schroed} satisfies $\psi \in L^\infty(0,T; \Sigma^m)$ and the functional $ \mathcal J(\alpha)$ is G\^{a}teaux-differentiable for all $t\in [0,T]$,  with
\begin{equation}
\label{eq:J_prime}
 \mathcal J'(\alpha) =  \re  \int_{\R^d} \conj{\varphi}(t,x) V(x) \psi(t,x) \;dx - 2 \frac{d}{dt} \left( \dot{\alpha}(t) \pare{\gamma_2 + \gamma_1 \omega^2(t)} \right),
\end{equation}
in the sense of distributions, where $\omega(t)$ is the weight factor defined in \eqref{eq:weight} and $\varphi\in C([0,T];L^2(\R^d))$ is the solution of the adjoint equation
\begin{equation}
\label{eq:adjoint}
 \begin{aligned}
        i \partial_t \varphi= &\  -\frac{1}{2} \Delta \varphi + U(x) \varphi + \alpha(t) V(x) \varphi + \lambda (\sigma+1) |\psi|^{2\sigma} \varphi + \lambda \sigma |\psi|^{2\sigma-2} \psi^2  \conj{\varphi}  \\
        & \ +
         4 \gamma_1 (\dot{\alpha}(t))^2 \, \omega(t) V(x) \psi,
       \end{aligned}
\end{equation}
subject to Cauchy data $\varphi(T, x) =   4i \langle \psi(T, \cdot), A\psi(T, \cdot) \rangle_{L^2_x}\, A\psi(T, x)$.
\end{theorem}

\begin{remark} When compared to the assumptions of Theorem \ref{thm:ex_min}, the result of Theorem \ref{thm:J_prime} 
requires additional regularity (and stronger decay) 
of the initial data $\psi_0$ and the potential $V$ (plus, we need to restrict ourselves to $\sigma \in \N$). 
Note that the requirement $m\in \N$ and $m \ge 2$ implies $m>d/2$ for $d=1,2,3$ spatial dimensions. 
\end{remark}

\begin{proof} We need to prove that $ \mathcal J'(\alpha)$ is of the form \eqref{eq:dJweak}. For this purpose, let $\psi=\psi(\alpha)$, $\tilde{\psi}={\psi} (\tilde{\alpha})$ with $\tilde \alpha = \alpha + \varepsilon \delta_\alpha$, satisfy the assumptions of 
Lemma~\ref{prop:lipschitz_in_alpha} and consider the difference of the corresponding objective functionals $\mathcal J(\alpha), \mathcal J(\tilde{\alpha})$. 
This difference can be written as the sum of three terms
\begin{align*}
 \mathcal J(\tilde{\alpha}) - \mathcal J(\alpha) =  \text{I} + \text{II} + \text{III} ,
 \end{align*}
where we define
\begin{align*}
 \text{I} := \langle \tilde{\psi}(T), A\tilde{\psi}(T) \rangle_{L^2_x}^2 - \langle \psi(T), A\psi(T) \rangle_{L^2_x}^2,\quad \text{II}  := \gamma_2 \int_0^T (\dot{\tilde{\alpha}}(t))^2 - (\dot{\alpha}(t))^2 \;dt, 
\end{align*}
and
\begin{align*} 
\text{III} := & \ \gamma_1 \int_0^T (\dot{\tilde{\alpha}}(t))^2 \left(\int_{\R^d} V(x) |\tilde{\psi}(t,x)|^2 \;dx \right)^2 dt \\
& \ - \gamma_1 \int_0^T (\dot{\alpha}(t))^2 \left(\int_{\R^d} V(x) |\psi(t,x)|^2 \;dx \right)^2 dt.
\end{align*}
The general strategy will be to use the Lipschitz property established in Lemma~\ref{prop:lipschitz_in_alpha} and rewrite the terms I, II, and III in such a way that
\begin{align*}
 \mathcal J(\tilde{\alpha}) - \mathcal J(\alpha) =  \mbox{linear terms in $(\tilde \alpha - \alpha)$} + \mathcal O(\|  \tilde \alpha - \alpha\|_{H^1_t}^2).
 \end{align*}
Since $\tilde \alpha = \alpha + \varepsilon \delta_\alpha$ and thus $ \mathcal O(\|  \tilde \alpha - \alpha\|_{H^1_t}^2) =\mathcal O(\varepsilon^2)$, 
the limit $\varepsilon \to 0$ then yields the desired functional derivative.
\smallskip

We start by considering the term I. It can be rewritten in the form
\begin{align*}
 \text{I} &= \langle \tilde{\psi}(T), A\tilde{\psi}(T) \rangle_{L^2_x}^2 - \langle \psi(T), A\psi(T) \rangle_{L^2_x}^2\\
 &= 2\langle \psi(T), A\psi(T) \rangle_{L^2_x}\left(\langle \tilde{\psi}(T), A\tilde{\psi}(T) \rangle_{L^2_x} - \langle \psi(T), A\psi(T) \rangle_{L^2_x} \right)\\&\qquad + \left(\langle \tilde{\psi}(T), A\tilde{\psi}(T) \rangle_{L^2_x} - \langle \psi(T), A\psi(T) \rangle_{L^2_x} \right)^2.
\end{align*}
Using the essential self-adjointness of $A$, the terms within the parentheses yield
 \begin{align*}
 &\langle \tilde{\psi}(T), A\tilde{\psi}(T) \rangle_{L^2_x} - \langle \psi(T), A\psi(T) \rangle_{L^2_x}\\ & \quad= 2 \langle \tilde{\psi}(T) - \psi(T), A\psi(T) \rangle_{L^2_x} + \langle \tilde{\psi}(T) - \psi(T), A(\tilde{\psi}(T) - \psi(T)) \rangle_{L^2_x}.
\end{align*}
Using the Lipschitz-estimate~\eqref{eq:lipschitz_in_alpha}, we obtain 
\[
 \left|\langle \tilde{\psi}(T) - \psi(T), A(\tilde{\psi}(T) - \psi(T)) \rangle_{L^2_x}\right| \le \|A\|_{\mathcal{L}(\Sigma,L^2_x)} \|\tilde{\psi}(T) - \psi(T)\|_{\Sigma}^2 \le C \varepsilon^2 \|\delta_\alpha\|_{H^1_t}^2,
\]
and hence
\begin{equation*}
 \langle \tilde{\psi}(T), A\tilde{\psi}(T) \rangle_{L^2_x} - \langle \psi(T), A\psi(T) \rangle_{L^2_x} = 2 \langle \tilde{\psi}(T) - \psi(T), A\psi(T) \rangle_{L^2_x} \\+ \mathcal O(\|\tilde{\alpha} - \alpha\|_{H^1_t}^2).
\end{equation*}
Squaring the above result and plugging it into our expression for $I$ consequently yields
\begin{equation}\label{eq:I1}
\text{I} = 4 \langle \psi(T), A\psi(T) \rangle_{L^2_x}\ \langle \tilde{\psi}(T) - \psi(T), A\psi(T) \rangle_{L^2_x} + \mathcal O(\|\tilde{\alpha} - \alpha\|_{H^1_t}^2).
\end{equation}
Next we consider II, which can be written as
\begin{align*}
\text{II} = & \ 2\gamma_2 \int_0^T \dot{\alpha}(t) \left(\dot{\tilde{\alpha}}(t)-\dot{\alpha}(t)\right) dt + \gamma_2 \int_0^T \left(\dot{\tilde{\alpha}} (t)- \dot{\alpha}(t)\right)^2 dt\\
= & \ 2\gamma_2 \int_0^T \dot{\alpha}(t) \left(\dot{\tilde{\alpha}}(t)-\dot{\alpha}(t)\right) dt + \mathcal O(\|\tilde{\alpha} - \alpha\|_{H^1_t}^2).
\end{align*}
The first term in the second line is thereby seen to be of the form given in \eqref{eq:dJweak}. Finally we consider III, which in view of definition \eqref{eq:weight} can be written as
\begin{align*}
\text{III} = &\  \gamma_1\int_0^T \left( (\dot{\tilde{\alpha}}(t))^2 - (\dot{\alpha}(t))^2 \right) \omega^2(t)\, dt  \\
&\ + \gamma_1\int_0^T (\dot{\tilde{\alpha}}(t))^2 \left( \left( \int_{\R^d} V(x) |\tilde{\psi}(t,x)|^2 dx \right)^2 - \omega^2(t) \right) dt .
\end{align*}
As before, we can expand these terms using quadratic expansions in both $\tilde{\psi}$ and $\tilde{\alpha}$. 
In view of the Lipschitz estimate \eqref{eq:lipschitz_in_alpha}, any quadratic error $\|\tilde{\psi} - \psi\|_{L^\infty_t L^2_x}^2$ is bounded by $\mathcal{O}(\|\tilde{\alpha} - \alpha\|_{H^1_t}^2)$ and hence we obtain
\begin{equation}
\label{eq:III}
\begin{split}
\text{III} =  & \ 4\gamma_1 \int_0^T (\dot{\alpha}(t))^2 \, \omega(t) \, \Big(\re \int_{\R^d} \left((\conj{\tilde{\psi}}-\conj{\psi })V\psi\right) (t,x)\;dx\Big)  dt \\
& + \  2\gamma_1 \int_0^T \left(\dot{\tilde{\alpha}}(t) - \dot{\alpha}(t)\right) \dot{\alpha}(t)\, \omega^2(t) \;dt +\mathcal O(\|\tilde{\alpha} - \alpha\|_{H^1_t}^2).
\end{split}
\end{equation}
Here the second term on the right hand side is linear in $(\tilde{\alpha} - \alpha)$ and hence of the desired form. In order to treat the first term, we note that the expression
\[
 4\gamma_1 \left( (\dot{\alpha}(t))^2 \int_{\R^d} V(x) |\psi(t,x)|^2 \;dx \right) V(x) \psi (t,x)
\]
appears as a source term in the adjoint equation~\eqref{eq:adjoint}. Thus we obtain 
\begin{equation}
 \label{eq:help_integral}
 \begin{split}
&\ 4\gamma_1 \int_0^T (\dot{\alpha}(t))^2 \, \omega(t) \, \Big(\re \int_{\R^d} \left((\conj{\tilde{\psi}}-\conj{\psi })V\psi\right) (t,x)\;dx\Big)  dt \\
 & \ = \re \int_0^T \int_{\R^d} \conj{\varphi}(t,x) \left( \partial_\psi P(\psi,\alpha)(\tilde{\psi}-\psi) \right)(t,x) \,  dx\\
 &\quad \ - \re  \int_{\R^d} i\, \conj{\varphi}(T,x) \big(\tilde{\psi}(T,x)-\psi(T,x)\big)dx,
\end{split}
\end{equation}
where we recall that $\partial_\psi P(\psi,\alpha)$ denotes the linearized Schr\"odinger operator obtained in \eqref{eq:linNLS}.
The last term on the right hand side of \eqref{eq:help_integral} stems from the boundary condition at $t=T$. Note that the boundary term at $t=0$ vanishes since $\tilde{\psi}(0)=\psi_0=\psi(0)$ by assumption. 
We recall that $\varphi\in C([0,T];L^2(\R^d))$ and 
\[
 \tilde{\psi}, \psi \in L^\infty(0,T;\Sigma^m) \cap W^{1,\infty}(0,T;\Sigma^{m-2}), \quad \text{with $m\ge 2$},
\]
and hence the right hand side of \eqref{eq:help_integral} is well-defined. 
In addition, since both $\tilde{\psi}$ and $ \psi$ solve the nonlinear Schr\"odinger equation \eqref{eq:schroed}, we can write
\begin{equation}
\label{eq:linearized_III}
 \begin{split}
  \partial_\psi P(\psi,\alpha) (\tilde{\psi}-\psi) = & \ i \partial_t (\tilde{\psi}-\psi) - H (\tilde{\psi}-\psi) - V(x)( \tilde{\alpha}(t) \tilde{\psi}-\alpha(t) \psi ) \\
  & \ + \lambda |\psi|^{2\sigma} \psi - \lambda |\tilde{\psi}|^{2\sigma} \tilde{\psi}  + (\tilde{\alpha}(t)-\alpha(t)) V(x) \tilde{\psi} + \varrho(\tilde{\psi},\psi) \\
  = & \ (\tilde{\alpha}(t)-\alpha(t)) V (x)\tilde{\psi} + \varrho(\tilde{\psi},\psi),
 \end{split}
\end{equation}
where the remainder $ \varrho(\tilde{\psi},\psi)$ is given by
\[
 \frac{1}{\lambda} \varrho(\tilde{\psi},\psi) = |\tilde{\psi}|^{2\sigma}\tilde{\psi} - |\psi|^{2\sigma}\psi - (\sigma+1) |\psi|^{2\sigma} (\tilde{\psi}-\psi) - \sigma |\psi|^{2\sigma-2}\psi^2\left(\conj{\tilde{\psi}}-\conj{\psi}\right).
\]
Since $2\sigma\ge 2$ by assumption, $\|\tilde{\psi}\|_{L^\infty_tL^\infty_x}, \|\psi\|_{L^\infty_t L^\infty_x} \le C$ in view of Corollary~\ref{cor:psi_bound}, and $\Sigma^m \hookrightarrow L^\infty(\R^d)$, the remainder can be bounded by
\[
| \varrho(\tilde{\psi},\psi)|\le C\left( |\tilde{\psi}|^{2\sigma-1}+|\psi|^{2\sigma-1} \right) |\tilde{\psi}-\psi|^2\le C |\tilde{\psi}-\psi|^2.
\]
In addition, since $\varphi\in C([0,T];L^2(\R^d))$ and $\Sigma^m\subset H^m(\R^d)\hookrightarrow L^4(\R^d)$, we find that
\[
\int_{\R^d} |\varphi(t,x)| |\tilde{\psi}(t,x)-\psi(t,x)|^2 \;dx \le \|\varphi\|_{L^\infty_tL^2_x} \| \tilde{\psi} - \psi \|_{L^\infty_tL^4_x}^2  = O(\|\tilde{\alpha} - \alpha\|_{H^1_t}^2).
\]
Furthermore, the contribution of $(\tilde{\alpha}(t)-\alpha(t)) V(x) \tilde{\psi}$ in~\eqref{eq:linearized_III} equals
\[
 (\tilde{\alpha}(t)-\alpha(t)) V (x)\psi + (\tilde{\alpha}(t)-\alpha(t)) V(x) (\tilde{\psi} - \psi),
\]
where the latter term can be estimated by $O(\|\tilde{\alpha} - \alpha\|_{H^1_t}^2)$ as before.
In summary, this shows that  
\begin{equation}
\label{eq:help_integral2}
\begin{split}
\eqref{eq:help_integral}= & \int_0^T (\tilde{\alpha}(t) - \alpha(t)) \re \int_{\R^d} \conj{\varphi}(t,x) V(x) \psi(t,x) dx dt  + O(\|\tilde{\alpha}-\alpha\|_{H^1_t}^2)\\
 & - 4 \langle \psi(T), A\psi(T) \rangle_{L^2_x}\ \langle \tilde{\psi}(T) - \psi(T), A\psi(T) \rangle_{L^2_x},
\end{split}
\end{equation}
where we have used the fact that the data of the adjoint problem at $t=T$ is given by 
\[
 \varphi(T,x) =   4i \langle \psi(T, \cdot), A\psi(T, \cdot) \rangle_{L^2_x}\ A\psi(T,x).
\]
Thus, we infer that, up to quadratic errors, the second line in~\eqref{eq:help_integral2} cancels with the terms obtained in \eqref{eq:I1}. 
Collecting all the expressions obtained for $\text{I, II, III}$ and taking the limit $\varepsilon \to 0$, we have shown that $\mathcal J(\alpha)$ is G\^{a}teaux-differentiable with derivative $\mathcal J'(\alpha)$ 
given by \eqref{eq:dJweak}. 
This concludes proof of Theorem \ref{thm:J_prime}. \end{proof}

Equation~\eqref{eq:J_prime} yields the following characterization of the critical points $\alpha_* \in H^1(0,T)$, i.e.\ points where $\mathcal J'(\alpha_*) = 0$. 

\begin{corollary} \label{cor:ODE}
Let $\psi_*$ be the solution of \eqref{eq:schroed} with control $\alpha_*$. Also, let $\varphi_*$ be the corresponding solution of the adjoint equation \eqref{eq:adjoint}, and denote by 
$\omega_*$ the function defined in \eqref{eq:weight} with $\psi$ replaced by $\psi_*$. Then $\alpha_* \in C^2(0,T)$ 
is a classical solution of the following ordinary differential equation
\begin{equation}\label{eq:ODEalpha}
 \frac{d}{dt} \left( \dot{\alpha}_*(t) \pare{\gamma_2 + \gamma_1 \omega_*^2(t)} \right) =  \frac{1}{2} \, \re \int_{\R^d} \conj{\varphi_*}(t,x) V(x) \psi_*(t,x) \;dx,
\end{equation}
subject to $\alpha_*(0) = \alpha_0$, $\dot{\alpha}_*(T) = 0$. 
\end{corollary}
\begin{remark} In the case $\gamma_1=0$ this simplifies to the expression used in the physics literature; cf.\ \cite{schm}.\end{remark}

\begin{proof}

Let $\mu \in C_0^\infty(0,T)$ be a test function with compact support in $(0,T)$. Then, Theorem \ref{thm:ex_min} and Theorem \ref{thm:J_prime} imply that there 
exists $\alpha_*\in H^1(0,T)$ such that $\mathcal J'(\alpha_*) = 0$, satisfying \eqref{eq:J_prime} in the sense of distributions, i.e. 
\begin{equation*}\label{eq:dJweak}
 \int_0^T \dot{\alpha}_*(t) \dot{\mu}(t) \pare{\gamma_2 + \gamma_1 \omega_*^2(t)}dt =  \frac{1}{2} \re \int_0^T  \int_{\R^d} \mu(t) \conj{\varphi_*}(t,x) V(x) \psi_*(t,x) dx dt,
\end{equation*}
where we have used the fact that the boundary terms at $t=0$ and $t=T$ vanish due to the compact support of $\mu(t)$.
We shall show that the weak solution $\alpha_*$ is in fact unique. This can be seen by considering two different $\alpha^1_*(t), \alpha_*^2(t)$, satisfying $\alpha^1_*(0)= \alpha_*^2(0)=\alpha_0$. 
Denoting their difference by 
$\beta_*= \alpha^1_* - \alpha_*^2$, we have that $\beta_*(t)$ solves
\[
 \int_0^T \dot{\beta}_*(t) \dot{\mu}(t) \pare{\gamma_2 + \gamma_1 \omega_*^2(t)}dt = 0, \quad \text{for all $\mu \in C_0^\infty(0,T)$.}
\]
Since $\gamma_2 > 0$ and $\gamma_1\ge 0$, this implies that $\dot \beta_*(t) = 0$ in the sense of distributions. However, since 
$\alpha^1_*, \alpha_*^2\in H^1(0,T) \hookrightarrow C(0,T)$, we conclude that $\beta_* \in C(0,T)$ and thus $\beta_* (t) = \text{const}$ for all $t\in [0,T]$. 
Since $\beta_*(0) = 0$ by assumption, we infer uniqueness of the weak solution $\alpha_*(t)$. 
On the other hand, standard arguments imply that \eqref{eq:ODEalpha} admits a unique classical solution $\alpha_* \in C^2(0,T)$, provided 
$\omega_* \in C^1(0,T)$ and the (source term on the) right hand side is continuous in time. The latter is obviously true in view of Proposition \ref{prop:NLS_sol} and Proposition \ref{prop:soladjoint}. 
In addition, since $V\in W^{1, \infty}(\R^d)$, we infer that for all $\psi(t) \in \Sigma$ it holds that $\chi(t):= (V(x)\psi(t)) \in \Sigma$. From Proposition \ref{prop:NLS_sol} it follows that
\[
\dot \omega_*(t) = 2 \re \int_{\R^d} V(x) \partial_t \psi_*(t,x) \conj{ \psi_*}(t,x) \, dx = 2 \langle \chi (t), \dot \psi_*(t) \rangle_{\Sigma, \Sigma^*}< +\infty.
\]
Thus, $\omega(t) \in C^1(0,T)$, yielding the existence of a unique classical solution $\alpha_*\in C^2(0,T)$. 
We therefore conclude that the unique weak solution $\alpha_*$ obtained above is in 
fact a classical solution, satisfying \eqref{eq:ODEalpha} subject to $\alpha_*(0) = \alpha_0$, $\dot{\alpha}_*(T) = 0$.
\end{proof}
We call $\alpha_*\in H^1(0,T)$ a {\it critical} or {\it stationary} point of the problem
\begin{equation}\label{Pcontrol}
\text{min}\quad\mathcal{J}(\alpha)\quad\text{over}\quad\alpha\in H^1(0,T)
\end{equation}
if $\mathcal{J}'(\alpha_*)=0$, where $\mathcal{J}'$ is given in Theorem~\ref{thm:J_prime}. In order the check computationally whether $\alpha_*$ is critical, one needs to solve \eqref{eq:schroed} for $\alpha=\alpha_*$ to obtain $\psi_*$ and then the adjoint equation \eqref{eq:adjoint}  with $\psi=\psi_*$ and $\alpha=\alpha_*$ to compute $\varphi_*$. Inserting $(\alpha,\psi,\varphi)=(\alpha_*,\psi_*,\varphi_*)$ in \eqref{eq:J_prime} yields $\mathcal{J}'(\alpha_*)$ which has to vanish for $\alpha_*$ to be critical, i.e., \eqref{eq:ODEalpha} is satisfied. We therefore call
\eqref{eq:schroed}, \eqref{eq:adjoint} and \eqref{eq:ODEalpha} the first order optimality conditions associated with \eqref{Pcontrol}.

\section{Numerical simulation of the optimal control problem} \label{sec:num}

For our numerical treatment we simplify to the case $d=\sigma = 1$. In this case, the first order optimality conditions for our optimal control problem are given by:
\begin{equation*}
\left \{
\begin{aligned}
& \frac{d}{dt} \left( \dot{\alpha}(t) \pare{\gamma_2 + \gamma_1 \omega^2(t)} \right) = \frac{1}{2} \, \re \int_{\R^d} \conj{\varphi(t,x)} V(x) \psi(t,x) \;dx,\\
& i \partial_t \psi +\frac{1}{2}\partial_x^2 \psi = (U(x) +  \alpha(t) V(x) )\psi +\lambda |\psi|^{2}\psi ,\\
 &i \partial_t \varphi  +\frac{1}{2} \partial_x^2 \varphi = (U(x) + \alpha(t) V(x) )\varphi  + 2 \lambda |\psi|^{2} \varphi +  \lambda\psi^2  \conj{\varphi}+
         4 \gamma_1 (\dot{\alpha})^2 \, \omega(t) V(x) \psi, \end{aligned}
\right. 
\end{equation*}
subject to the following conditions: $\alpha(0) = \alpha_0$, $\dot{\alpha}(T) = 0$, and 
\begin{equation*}
\psi(0,x) = \psi_0(x), \quad  \varphi(T,x) = 4i \langle \psi(T, \cdot), A\psi(T, \cdot) \rangle_{L^2_x} A\psi(T, x).
\end{equation*}

In our numerical simulations, the resulting Cauchy problems for Schr\"odinger-type equations are solved by a {\it time-splitting spectral method} 
of second order (Strang-splitting), as can be found in \cite{BJM}. 
This computational approach is unconditionally stable and allows for spectral accuracy in the resolution of the wave function $\psi(t,x)$. This is needed due to the 
highly oscillatory nature of solutions to (nonlinear) Schr\"odinger--type equations.
We consequently perform our simulations on a {\it numerical domain} $\Omega\subset \R$, equipped with periodic boundary conditions. The 
trapping potential $U(x)$ is thereby chosen such that the ``effective'' (i.e.\ the numerically relevant) support of the wave function $\psi(t,x)$ stays away from the boundary. 
In doing so, the boundary conditions do not significantly influence our results. 
A good test of the accuracy of our numerical code is given by the fact that the Gross-Pitaevskii equation conserves the physical mass (i.e. the $L^2$-norm of $\psi(t)$). 
Indeed, in all our numerical examples presented in Section \ref{sec:numex} below, we find that the $L^2$-norm is numerically preserved up to relative errors of the order $10^{-13}$.

\subsection{Gradient-related descent method} Once a suitable solver for the state and the adjoint equations is at hand, our gradient-related descent scheme operates as follows. We determine a
sequence of descent directions $(\delta_\alpha^k)\subset H^1(0,T)$, i.e., for every $k\in\mathbb{N}$
\[
 \mathcal{J}(\alpha_k + \delta_\alpha^k) < \mathcal{J}(\alpha_k)\equiv J(\psi(\alpha_k), \alpha_k).
\]
is satisfied. Note that a simple Taylor expansion of $\mathcal{J}$ around $\alpha_k$ shows that $\langle \mathcal{J}'(\alpha_k), \delta_\alpha^k \rangle<0$ is sufficient for $\delta_\alpha^k$ to be a descent direction for $\mathcal{J}$ at $\alpha_k$. We are in particular interested in gradient-related descent directions which satisfy
\[
M\delta_\alpha^k = -\mathcal{J}'(\alpha_k)
\]
with $M$ a suitably chosen positive definite operator.

A rather straightforward choice of $M$ is given by $M = \partial^2_\alpha J(\psi, \alpha_k)$. In this case $\delta_\alpha^k$ is obtained as the solution of the following ordinary differential equation (of second order):
\begin{align*}
\mathcal{J}'(\alpha_k) = 2 \frac{d}{dt} \left( \dot{\delta_\alpha^k}(t) \pare{\gamma_2 + \gamma_1 \left(\int_{\R} V(x) |\psi_k(t,x)|^2 dx\right)^2} \right),
\end{align*}
with $ \delta_\alpha^k(0) = 0$ and $\dot{\delta_\alpha^k}(T) = 0.$ Here $\psi_k(t,x)$ denotes the solution of the Gross--Pitaevskii equation with $\alpha(t) = \alpha_k(t)$.
With this choice of a descent direction, we then perform a \emph{line search} in order to decide on the length of the step taken along $\delta_\alpha^k$. In fact, we seek for $\nu_k > 0$ such that
\begin{equation}
 \label{eq:line_search}
 \mathcal J(\alpha_k + \nu_k \delta_\alpha^k) \leq \mathcal J(\alpha_k) + \mu\nu_k \langle \mathcal{J}'(\alpha_k), \delta_\alpha^k \rangle
\end{equation}
with some fixed $\mu\in (0,1)$.
Within each line search, we determine $\nu_k$ iteratively by a backtracking strategy. Thus, the whole procedure amounts to an Armijo line search method with backtracking. Of course, more elaborate strategies based on interpolation or alternative line search criteria are possible; see, e.g., \cite{NW_book} for more details.

We stop the gradient descent method whenever
\begin{equation}
\label{eq:term_cond}
 \|\mathcal J'(\alpha_k)\|_{H^{-1}_t} \leq \text{TOL} \cdot \|\mathcal J'(\alpha_1)\|_{H^{-1}_t}
\end{equation}
is satisfied for the first time. Here, $\textnormal{TOL} \in (0,1)$ is a given stopping tolerance and $\alpha_1 \in H^1(0,T)$ is the initial guess satisfying the boundary conditions $\alpha_1(0) = \alpha_0$ and $\dot{\alpha}_1(T) = 0$. As a safeguard, also an upper bound on the number of iterations is implemented.

In our tests, we observe the usual behavior of steepest descent type algorithms, i.e., the method exhibits rather fast progress towards a stationary point in early iterations, but then suffers from
scaling effects reducing the convergence speed. Therefore, often the maximum number of iterations is reached. Thus, we connect
the first-order, gradient method to a Newton-type method which relies on second derivatives
or approximations thereof.  

\subsection{Newton method}
The majority of iterations within our simulations are performed via a second order method, \emph{Newton's method}, for which we use the full Hessian
\[
M := D^2_{\alpha} \mathcal{J}(\alpha_k) : H^1(0,T)\times H^1(0,T) \to \R,
\]
or a sufficiently close positive definite approximation thereof. Note that we can also consider the Hessian as a map $D^2_\alpha \mathcal{J} : H^1(0,T) \to H^1(0,T)^*$. Recall that the gradient-related method above simply uses $M= \partial^2_{\alpha} J(\psi, \alpha_k)$. 

We derive $D^2_{\alpha} \mathcal{J}$ formally form the Lagrangian formulation; see Remark \ref{rem:lag}. The Lagrangian is given by
\[
L(\psi,\alpha,\varphi) = J(\psi,\alpha) - \langle \varphi, P(\psi,\alpha) \rangle_{L^2_{t,x}},
\]
where $\varphi$ is the solution to the adjoint equation \eqref{eq:adjoint} and $P(\psi,\alpha)$ is the Gross--Pitaevksii operator written in abstract form. Proceeding formally, we find 
\begin{align*}
 \langle (D^2_{\alpha}\mathcal{J}) \delta_\alpha, \tilde \delta_\alpha \rangle_{L^2_t} &= \langle( \partial^2_\psi L) \delta_\psi, \tilde \delta_\psi \rangle_{L^2_{t,x}} + \langle (\partial_{\psi \alpha} L) \delta_\alpha, \tilde \delta_\psi \rangle_{L^2_{t,x}}\\&\qquad + \langle (\partial_{\alpha \psi} L) \tilde\delta_\alpha, \delta_\psi \rangle_{L^2_t} + \langle (\partial^2_\alpha L) \delta_\alpha, \tilde \delta_\alpha \rangle_{L^2_t},
\end{align*}
where $\delta_\psi$ and $\tilde \delta_\psi$ solve the linearized Gross--Pitaevksii equation with controls $\delta_\alpha, \tilde \delta_\alpha$, respectively. 
In view of the derivation given in Section~\ref{sec:ident} we have 
\[
 \delta_\psi = \psi'(\alpha)\delta_\alpha = -\partial_\psi P(\psi(\alpha),\alpha)^{-1}\partial_\alpha P(\psi(\alpha),\alpha) \delta_\alpha,
\]
and analogously for $\tilde \delta_\psi$. Hence we conclude that
\begin{equation}
 \label{eq:hess}
\begin{aligned}
 & \ \langle (D^2_{\alpha}\mathcal{J}) \delta_\alpha, \tilde \delta_\alpha \rangle_{L^2_t}  = 
 \langle (\partial^2_\psi J) \psi'(\alpha)\delta_\alpha, \psi'(\alpha)\tilde \delta_\alpha\rangle_{L^2_{t,x}} + \langle (\partial_{\alpha\psi}J) \delta_\alpha, \psi'(\alpha) \tilde \delta_\alpha\rangle_{L^2_{t,x}}\\
 & \ + \langle (\partial_{\psi\alpha}J) \psi'(\alpha) \delta_\alpha, \tilde \delta_\alpha\rangle_{L^2_t} + \langle (\partial^2_\alpha J) \delta_\alpha,\tilde \delta_\alpha\rangle_{L^2_t} - 
 \langle \varphi , \big(D_\alpha^2 P (\psi, \alpha) \delta_\alpha\big) \tilde\delta_\alpha \rangle_{L^2_{t,x}},
 \end{aligned}
 \end{equation}
 where 
\begin{align*}
& \ (D_\alpha^2 P (\psi, \alpha) \delta_\alpha) \tilde\delta_\alpha = \big(\partial^2_\psi P(\psi, \alpha) (\psi'(\alpha)\delta_\alpha)\big)  (\psi'(\alpha) \tilde\delta_\alpha)\\
&\  + \big(\partial_{\alpha\psi} P(\psi, \alpha) \delta_\alpha\big) (\psi'(\alpha) \tilde \delta_\alpha) + \big(\partial_{\psi\alpha} P(\psi, \alpha)(\psi'(\alpha) \delta_\alpha)\big)\tilde \delta_\alpha,
\end{align*}
since $ \partial^2_\alpha P(\psi, \alpha) = 0$. All of the terms appearing on the right hand side of \eqref{eq:hess} can be evaluated by replacing $ \psi'(\alpha)\delta_\alpha$ by $-\partial_\psi P(\psi,\alpha)^{-1}\partial_\alpha P(\psi,\alpha)\delta_\alpha$. Consequently for calculating the action of the Hessian this requires to solve several linearized Schr\"odinger-type equations with different source terms and boundary data.
For example, the term involving $(\partial_{\alpha\psi}J)$ can be evaluated by using
\[
 \chi:=\partial_\psi P(\psi, \alpha) ^{-*} ((\partial_{\alpha\psi} J )\delta_\alpha),
\]
which solves the following Cauchy problem
\begin{align*}
  i \partial_t \chi + \frac{1}{2} \partial^2_x \chi = U(x) \chi + \alpha(t) V(x) \chi +  2\lambda |\psi|^{2} \chi + \lambda \psi^2  \conj{\chi} 
  + 8 \gamma_1 h(t,x) \psi,
\end{align*}
where $h(t,x): = \omega(t) \dot{\alpha}(t) \dot{\delta}_\alpha (t) V(x)$ and
$$\chi(T,x) = \frac{\delta^2J(\psi, \alpha)}{\delta\psi(T, x)\, \delta\alpha(T)} = 0.$$

Bearing this in mind, we have to solve the following equation for $\delta_\alpha^k \in H^1(0,T)$:
\begin{equation}
\label{eq:newt}
M \delta_\alpha^k\equiv  D^2_\alpha \mathcal{J} \delta_\alpha^k = -\mathcal{J}'(\alpha^k) \in H^{1}(0,T)^*.
\end{equation}
Hence, we need to invert $D^2_\alpha \mathcal{J}$, which, in view of \eqref{eq:hess} is not directly possible. Rather we resort to an iterative method, the \emph{preconditioned MINRES algorithm}, see~\cite{PS}, with the preconditioner $(\partial_\alpha^2 J(\psi, \alpha))^{-1} : H^1(0,T)^* \to H^1(0,T)$. 

We emphasize that here we aim to study the behavior of solutions of our control problem rather than at optimizing the respective solution algorithm or its implementation.

\subsection{Numerical examples}\label{sec:numex}
In all our examples, we choose the numerical domain $\Omega = [-L, L]$ with $L = 20$ and periodic boundary conditions. The number of spatial grid points is $N=256$. In addition, we set the final control time to be $T = 10$, and we use $M=1024$ equidistant time steps. In order to avoid the influence of the boundary, we choose a trapping potential $U(x) = 30\left(\frac{x}{L}\right)^2$. The initial guess for the control is taken to be just $\alpha_1 \equiv 0$ in the linear case ($\lambda = 0$), whereas each algorithm in the nonlinear case ($\lambda \neq 0$) is started from the control obtained by solving the linear problem.
In our tests of the first-order gradient method, we choose $\textnormal{TOL} = 10^{-8}$ in the terminating condition~\eqref{eq:term_cond} for the whole algorithm, $\mu = 10^{-3}$, and a maximum number of $20000$ iterations. For the Newton method, we likewise set $\textnormal{TOL} = 10^{-8}$ and we stop the algorithm after at most $45$ Newton steps.

\subsubsection{Example: shifting a linear wave packet} 
For validation purposes, we consider the time-evolution of a {\it linear} wave packet, i.e.\ $\lambda = 0$, whose center of mass we aim to shift towards a prescribed point $y_1 \in [-L, L]$.
For this purpose consider a control potential $$V(x) =  \frac{3}{10} + \frac{3x}{200}\ge 0,\quad \forall x \in [-L,L],$$ and the observable
 \[A(x) = 1-e^{-(\kappa(x-y_1))^2/L^2}.\] 
In this case, we find that the algorithm converges well even if we only invoke the first order gradient method. 
Indeed, as we decrease the regularization parameters $\gamma_1, \gamma_2 \ll 1$, we approach an optimal solution which, as it seems, cannot be improved upon. 
This optimal solution, or, more precisely, its spatial density $\rho = |\psi|^2$, is depicted in Figure~\ref{fig:shift_wave_pkt} (right plot), where we denote by ``target'' the function proportional to $1-A(x)$ with $\kappa=0.07$ and $y_1=-2L/8$, such that it has the same $L^2$--norm as $\psi_0$. The left plot shows the associated control.\\
\begin{figure}[ht]
  \centering
  \subfloat{\includegraphics[width=0.5\textwidth]{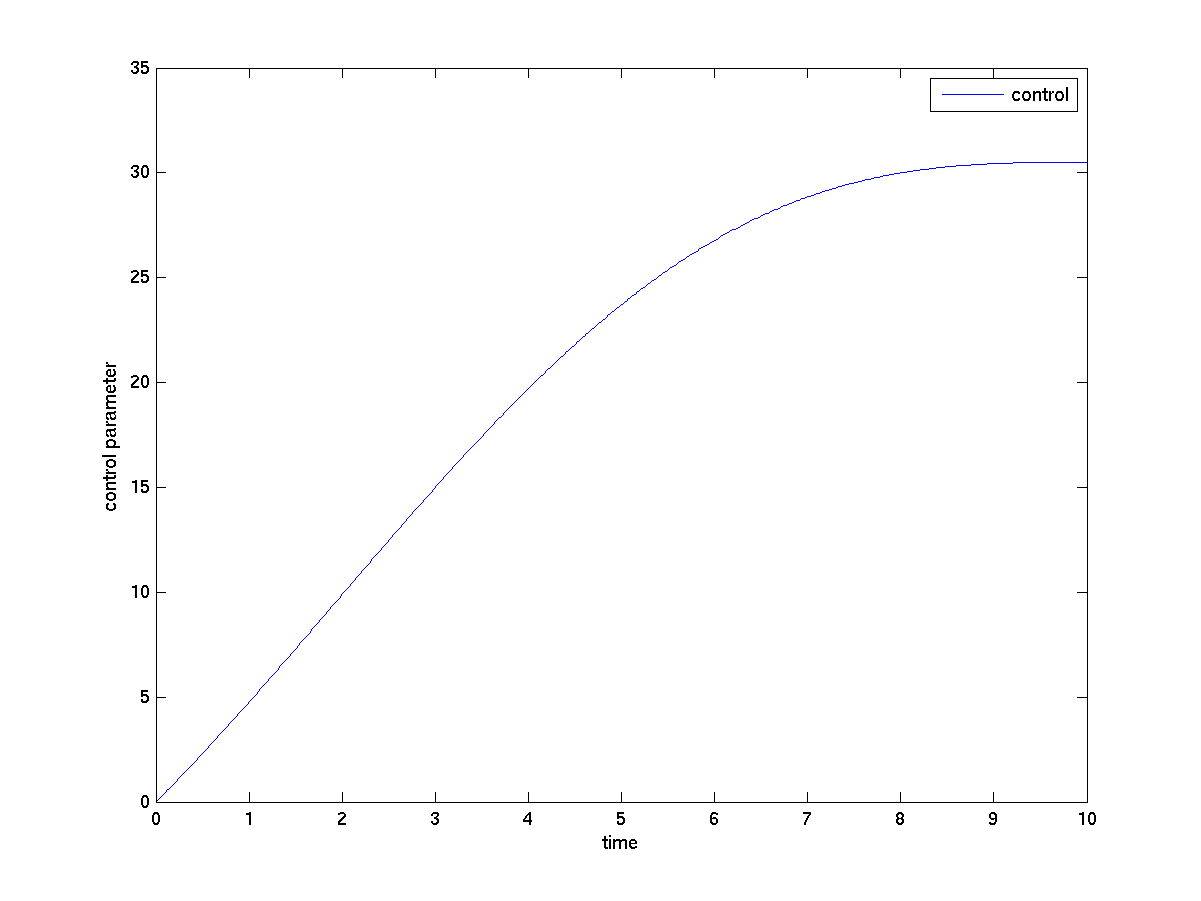}}
  \subfloat{\includegraphics[width=0.5\textwidth]{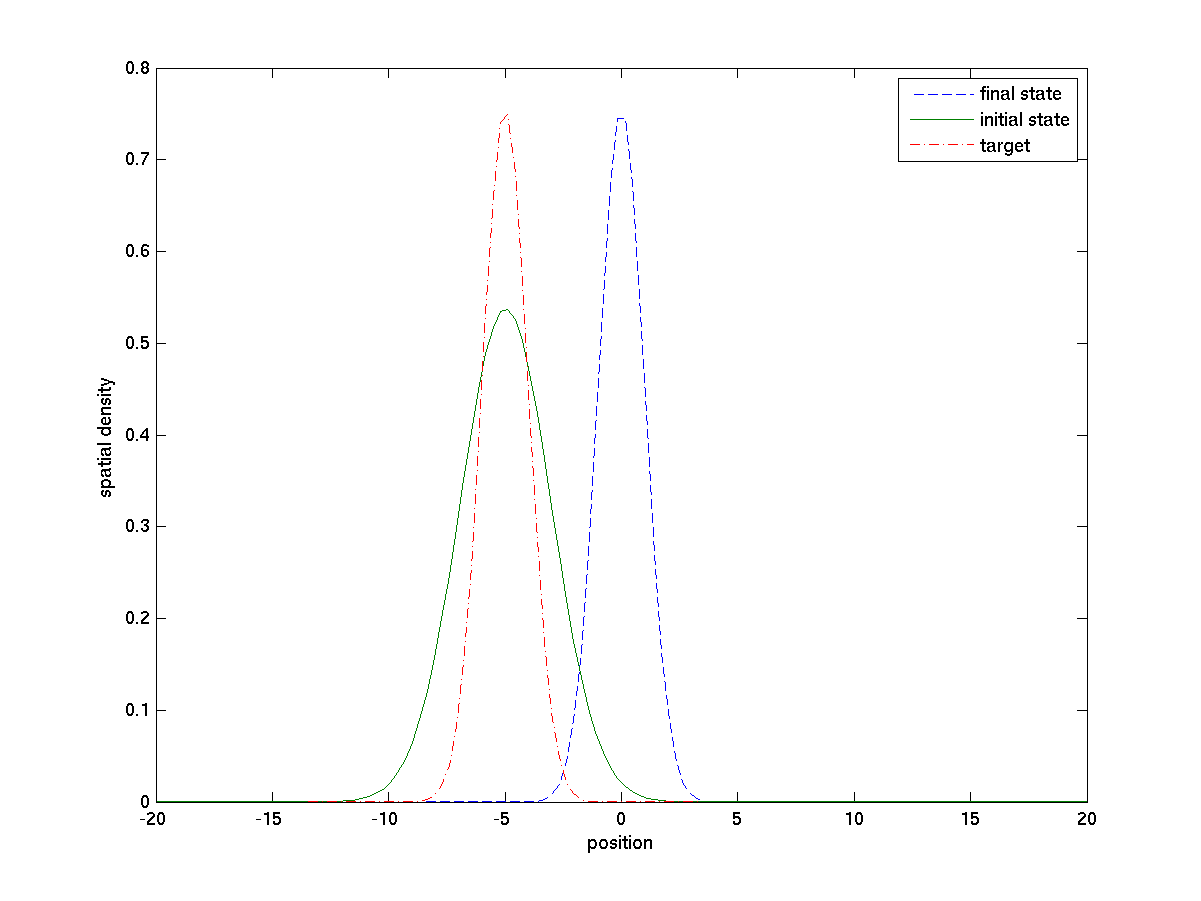}}
  \caption{Shifting a linear wave paket}
  \label{fig:shift_wave_pkt}
\end{figure}

Since this solution seems optimal, the choice of $\gamma_1, \gamma_2$ becomes negligible below a certain threshold. Thus, it suffices to consider $\gamma_1 = 0$ and only include the cost term 
proportional to $\gamma_2$. Similar results hold for any other given point $y_1\in \Omega$, provided $y_1$ stays sufficiently far away from the boundary.

\subsubsection{Example: splitting a linear wave paket}\label{ex:splitlin} 
We still consider the linear case, i.e., $\lambda = 0$, and aim to split a given initial wave packet into two separate packets centered around $y_1$ and $y_2$, respectively. The control potential is chosen as 
$$V(x) = e^{-8x^2/L^2} \ge 0,$$ 
and the observable  
\[A(x) =  1-\left(e^{-(\kappa(x-y_1))^2/L^2}+ e^{-(\kappa(x-y_2))^2/L^2}\right).\]  
In the following we fix $\kappa=0.07$, $y_1=-2L/8$, and $y_2=2L/8$.
In this case we find that the residual of the first order gradient method does not drop below the tolerance given in~\eqref{eq:term_cond} before the maximum number of iterations is reached.
With the Newton method, however, we find a (local) minimum of the objective functional $J(\psi, \alpha)$ in less than 20 Newton iterations. Of course there is no guarantee that this is a global minimum.

In order to illustrate our results, we consider the case where $\gamma_1 = 0$, $\gamma_2 = 1.5\times 10^{-6}$. At the final control time $T=10$ we then obtain: 
\[ \langle A\psi(T),\psi(T)\rangle_{L_x^2}^2 \approx 2.261 \times 10^{-3} .\]
The spatial density $\rho = |\psi|^2$ of the corresponding solution is shown in the right plot of Figure~\ref{fig:split_wave_pkt}. The associated control is depicted in the left plot.
\begin{figure}[ht]
  \centering
  \subfloat{\includegraphics[width=0.5\textwidth]{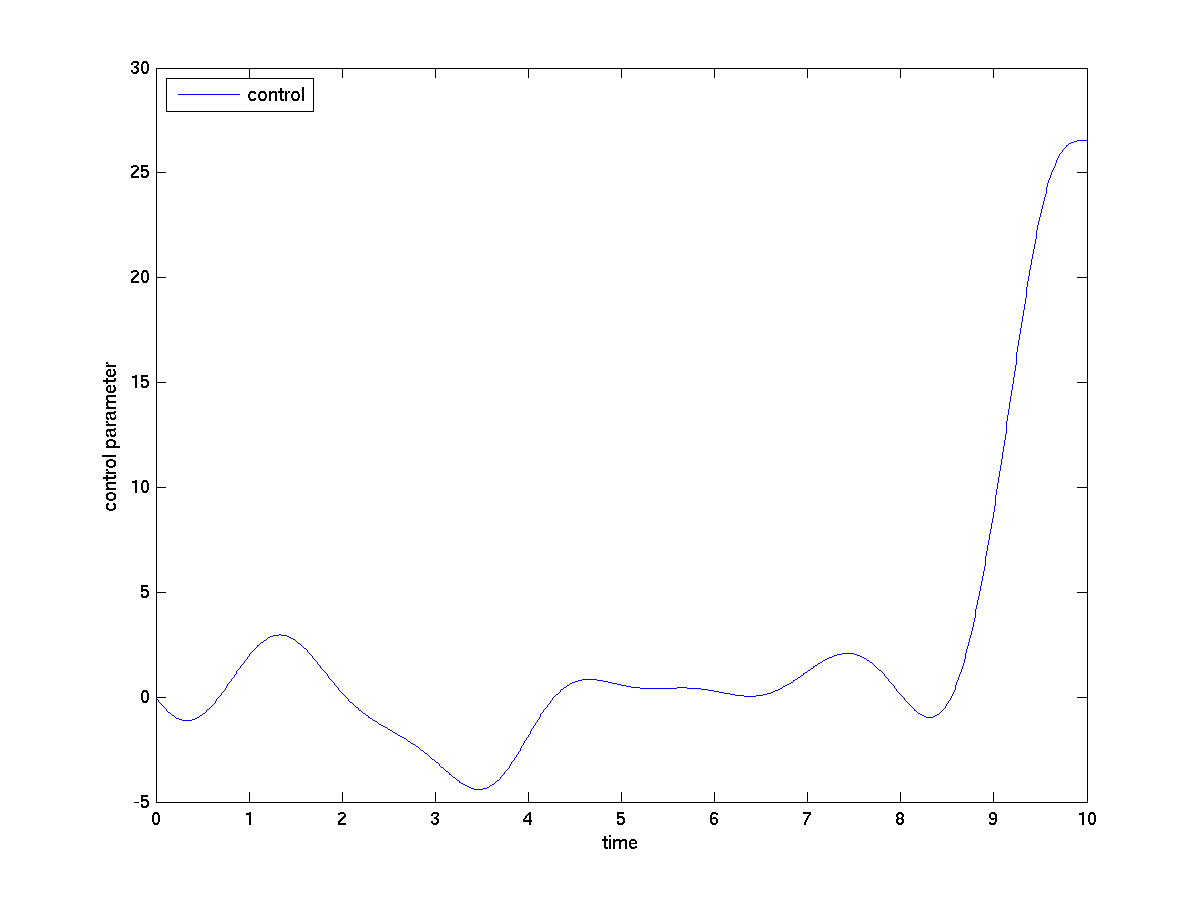}}
  \subfloat{\includegraphics[width=0.5\textwidth]{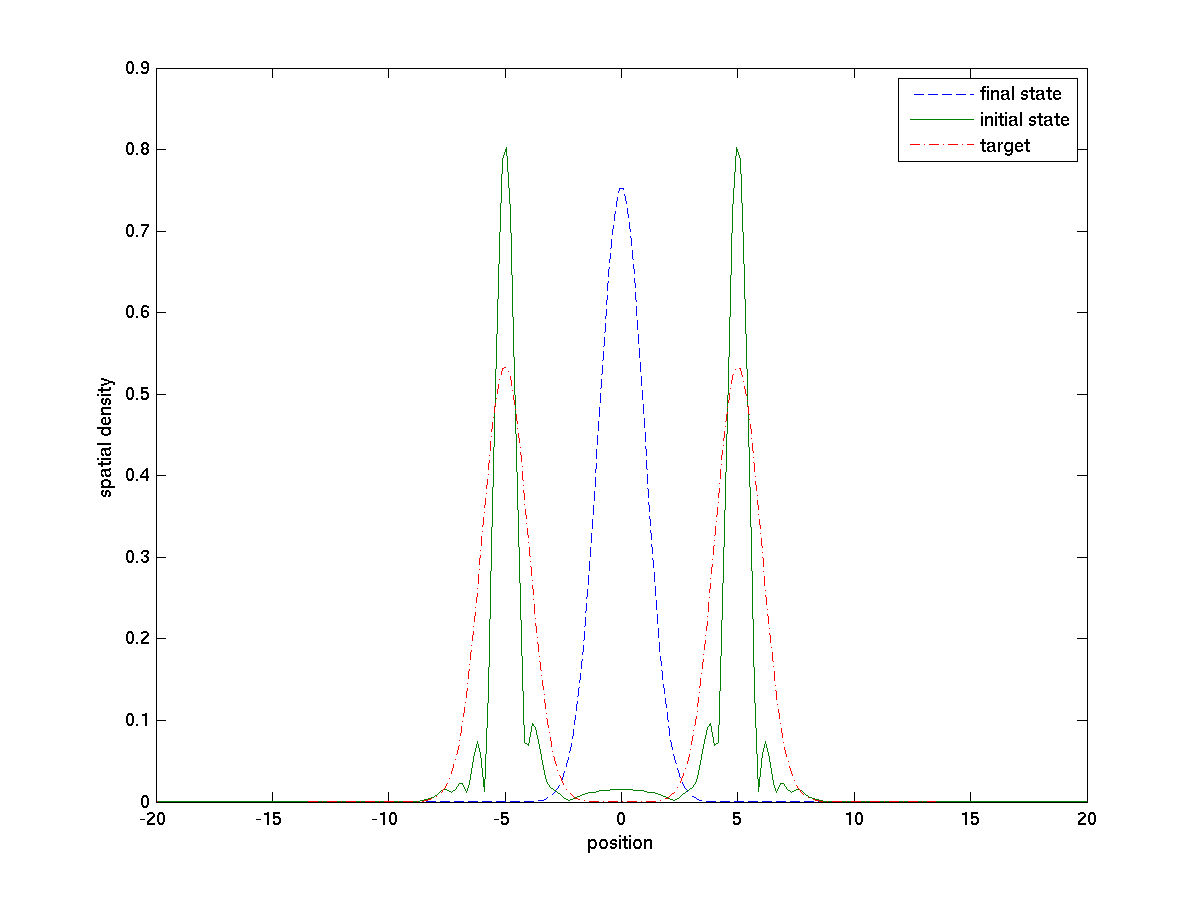}}
  \caption{Splitting a linear wave paket with $\gamma_1 = 0$}
  \label{fig:split_wave_pkt}
\end{figure}
If, instead, we choose $\gamma_1 = 4\times 10^{-5}$, $\gamma_2 = 1\times 10^{-9}$, we find 
\[\langle A\psi(T),\psi(T)\rangle_{L_x^2}^2 \approx 2.269\times 10^{-3},\]
and the corresponding solution is given in Figure~\ref{fig:split_wave_pkt2}. Here the intermediate state is a plot of $\rho(t)$ at $t=4=0.4\times T$.
\begin{figure}[ht]
  \centering
  \subfloat{\includegraphics[width=0.5\textwidth]{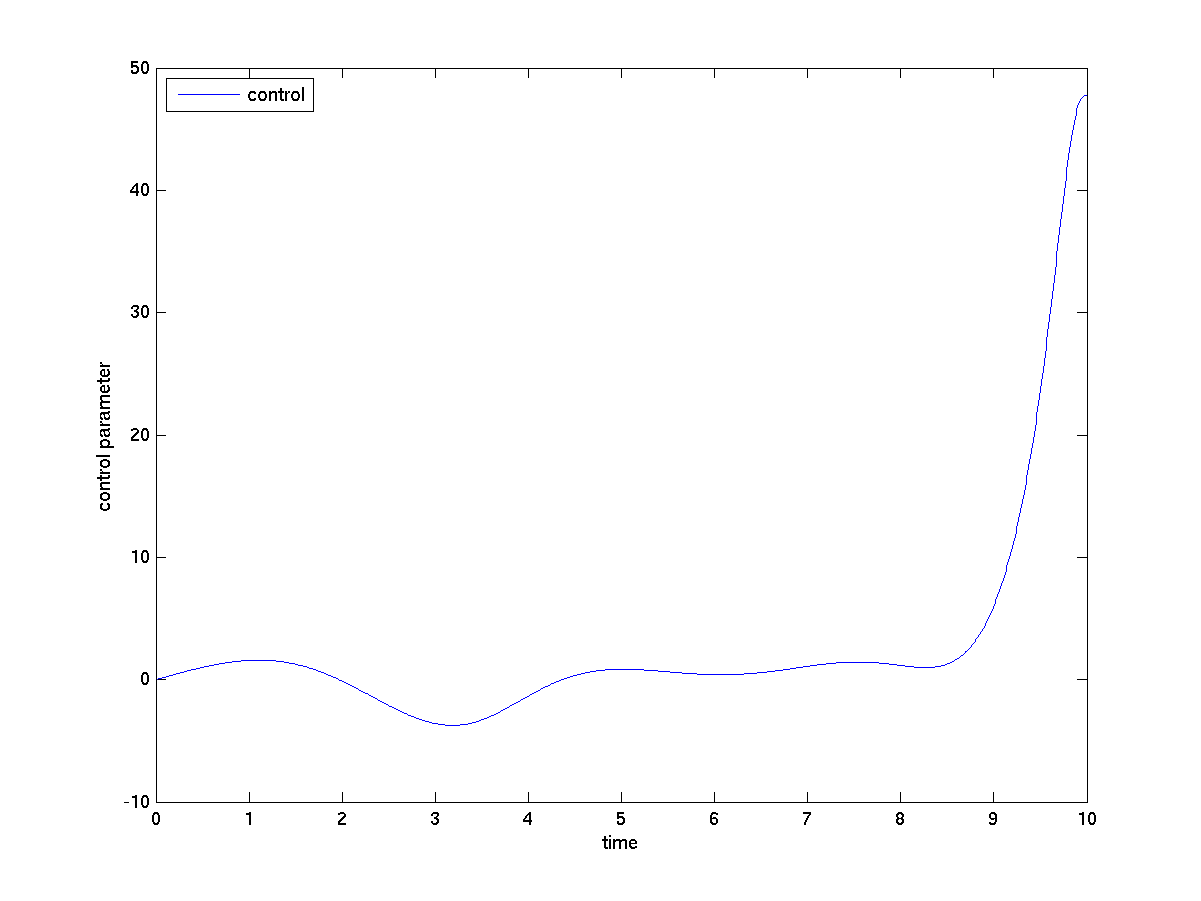}}
  \subfloat{\includegraphics[width=0.5\textwidth]{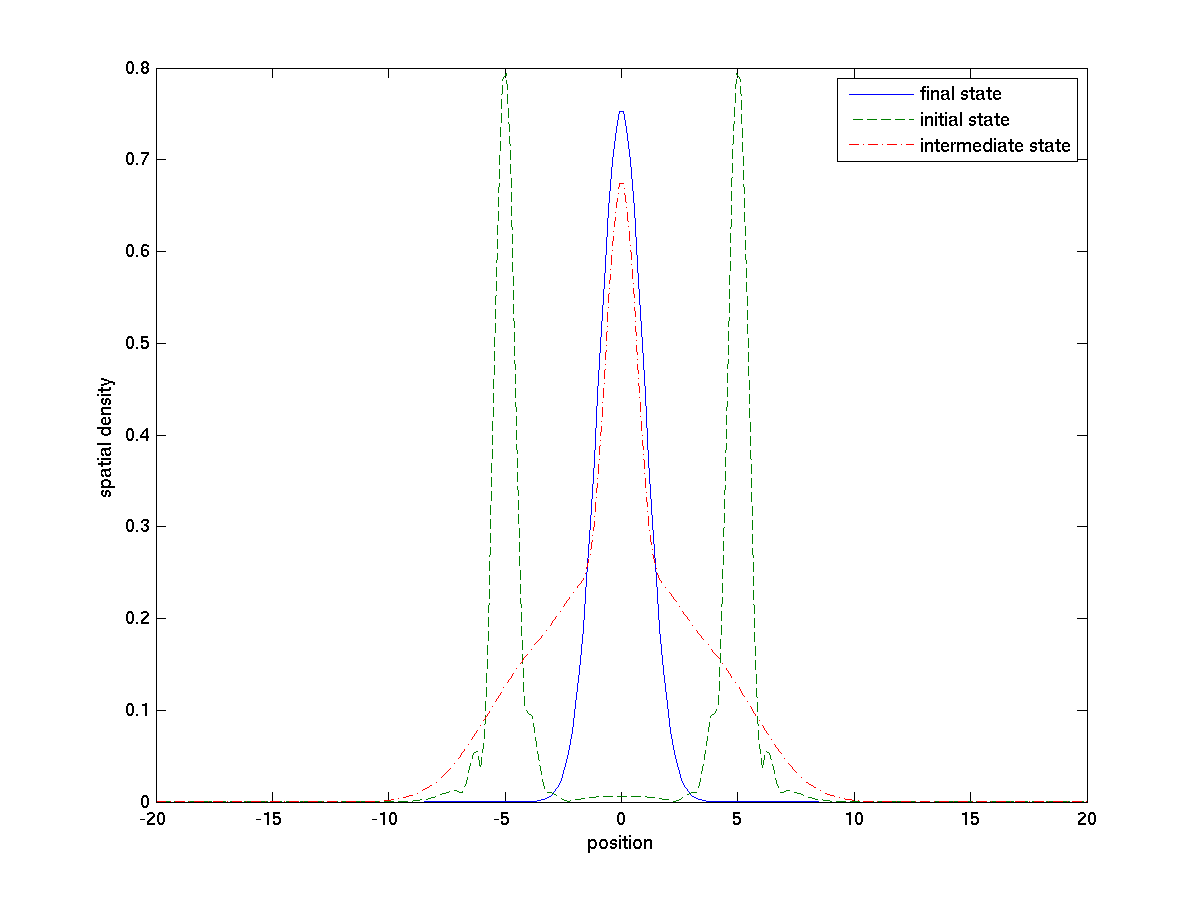}}
  \caption{Splitting a linear wave paket with $\gamma_1 > 0$}
  \label{fig:split_wave_pkt2}
\end{figure}

A direct comparison of the (spatial densities of the) resulting wave functions and the respective controls is given in Figure~\ref{fig:direct_comp}.
\begin{figure}[ht]
  \centering
  \subfloat{\includegraphics[width=0.5\textwidth]{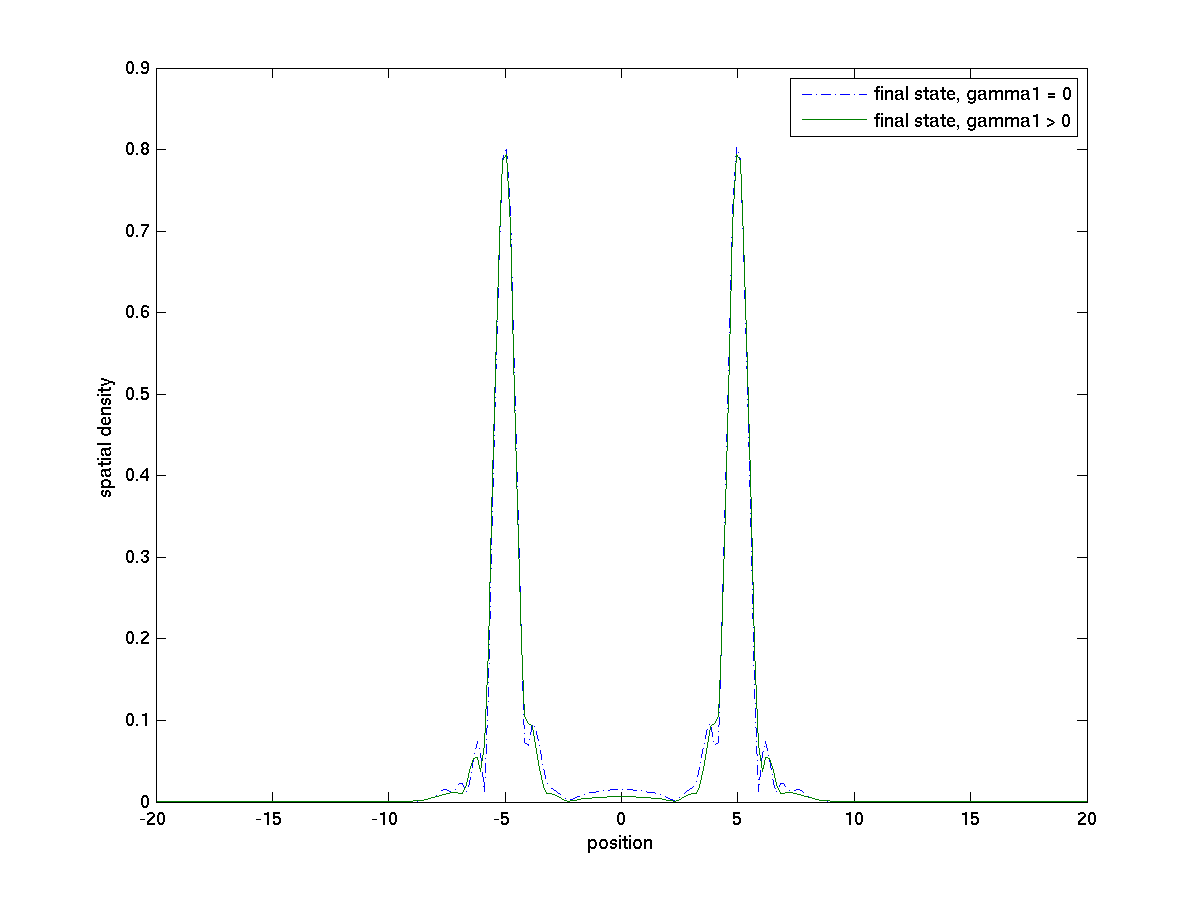}}                
  \subfloat{\includegraphics[width=0.5\textwidth]{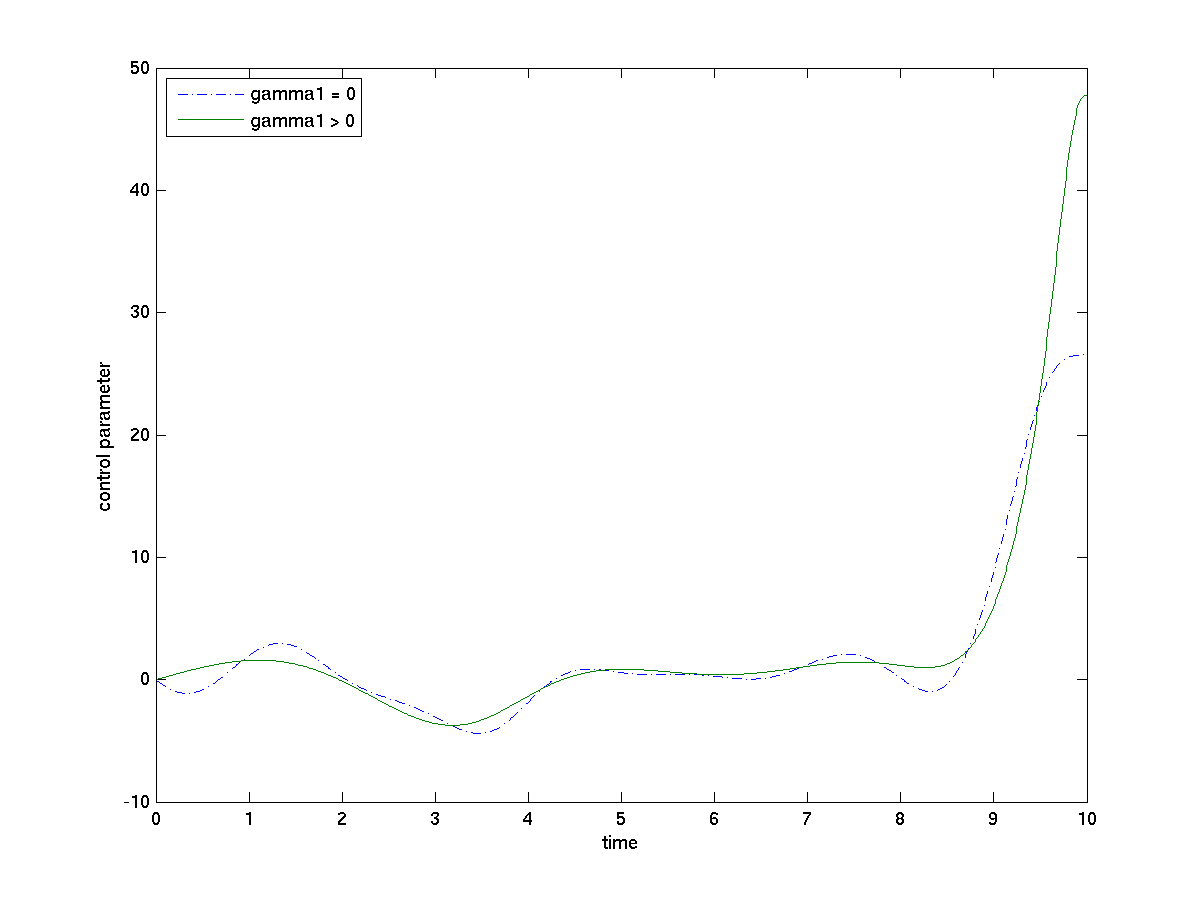}}
  \caption{Direct comparison between results}
  \label{fig:direct_comp}
\end{figure}
We see that the spatial densities are nearly identical, but the variability of the respective control parameters is not the same.
This is, of course, related to time--evolution of the weight factor $\omega(t)$, defined in \eqref{eq:weight}, which is shown in Figure~\ref{fig:int_V_psi} for the case of $\gamma_1 = 4\times 10^{-5}$ and $\gamma_2 = 1\times 10^{-9}$.\\
\begin{figure}[ht]
\begin{center}\includegraphics[width=0.5\linewidth]{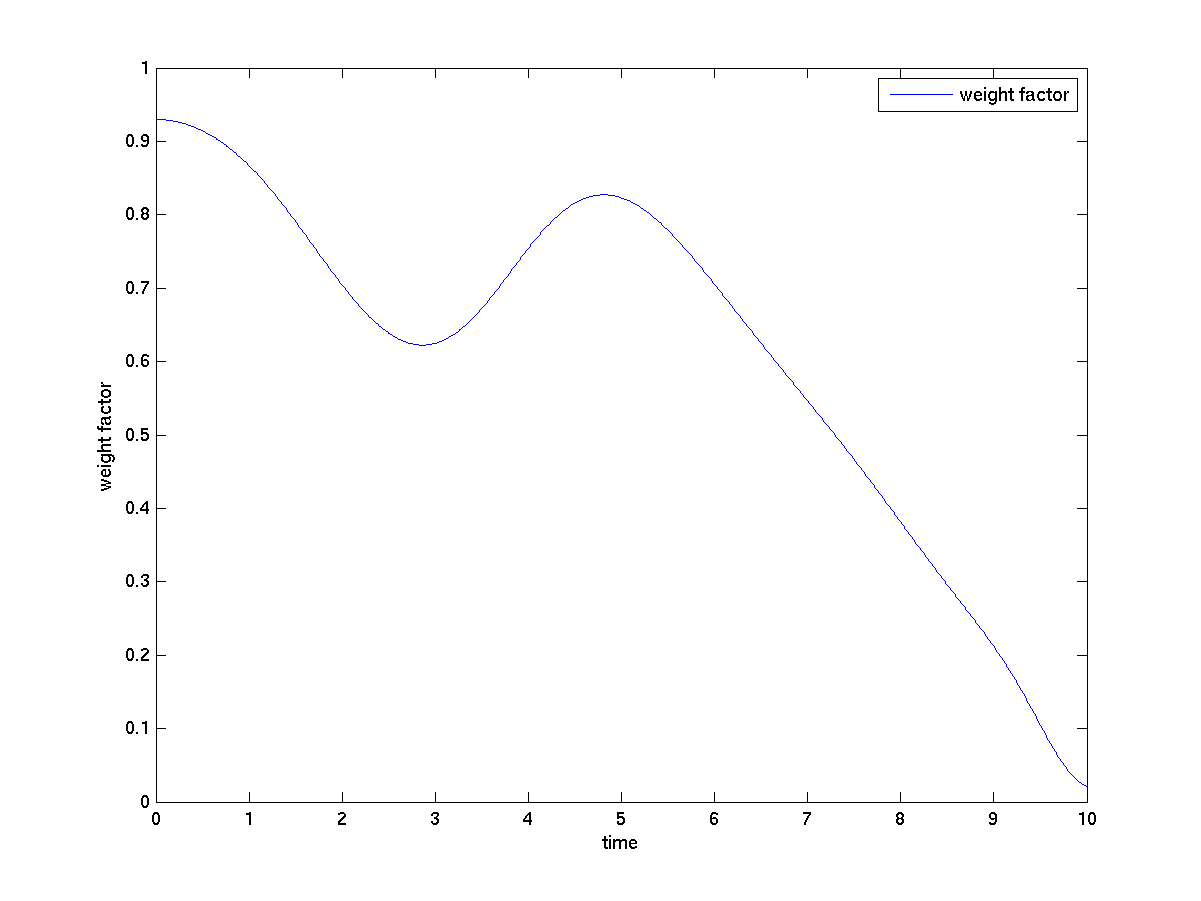}\end{center}
\caption{The weight factor $\omega = \int V |\psi|^2dx$ over time}
\label{fig:int_V_psi}
\end{figure}
By construction, the time--integral of $\omega(t)$ can be interpreted as the physical work performed during the control process. We find that compared to the case $\gamma_1>0$, the term $\|E(\cdot)\|_{L^2_t}^2$ is around $30\%$ larger (64.5 versus 49.1) and $\|\dot E (\cdot)\|_{L^2_t}^2$ is around twice as large (95.0 versus 43.4) in the case where $\gamma_1 = 0$, yielding a significant advantage of our control cost over terms considering the $H^1$-norm only; see \cite{schm} for the latter.

Finally, Figure~\ref{fig:J} shows an example of the evolution of the objective functional $J(\psi, \alpha)$ over the number of iterations of the Newton method, here for the case where $\gamma_1 = 0$.\\
\begin{figure}[ht]
\begin{center}\includegraphics[width=0.5\linewidth]{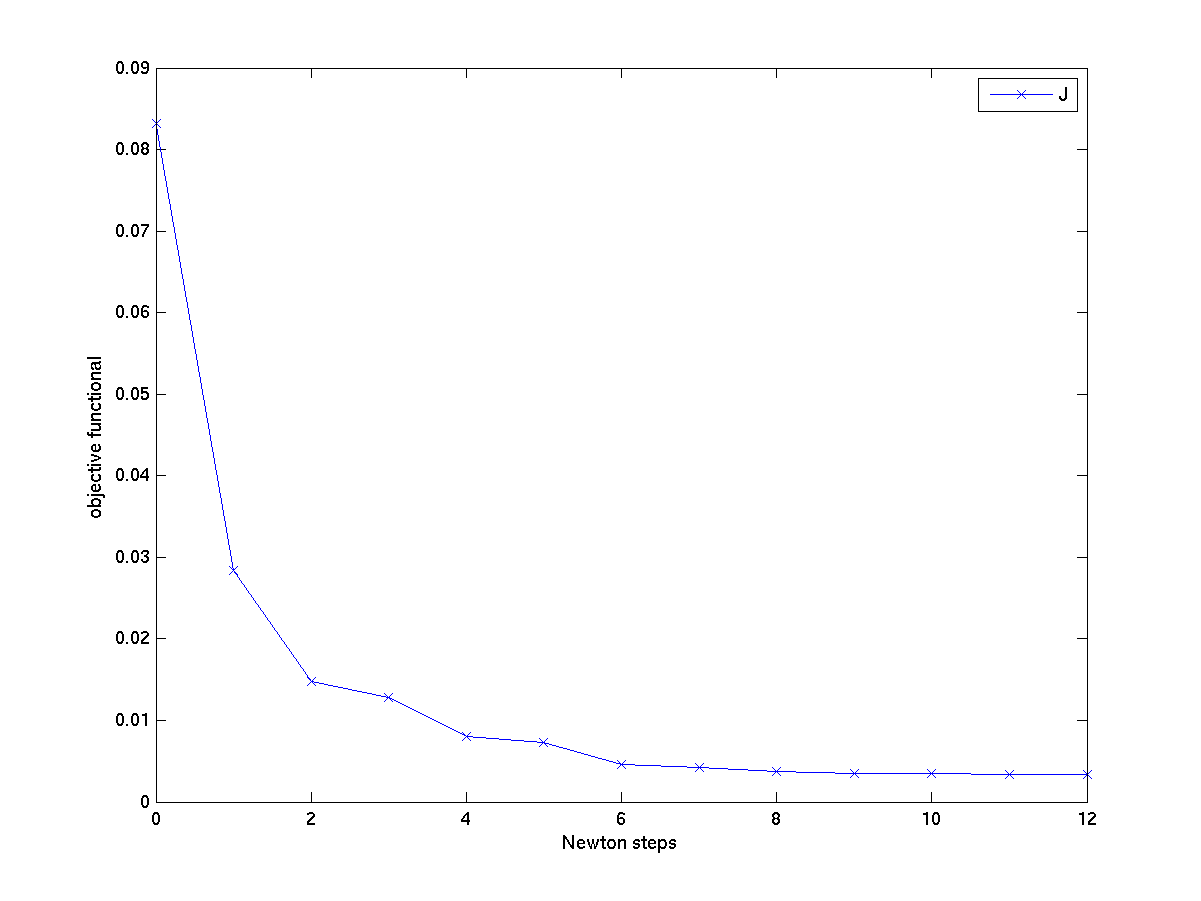}\end{center}
\caption{Value of $J(\psi, \alpha)$ over number of iterations}
\label{fig:J}
\end{figure}

\subsubsection{Example: splitting a Bose--Einstein condensate}
We consider the same situation as in the previous example, but with an additional (cubic) nonlinearity. More precisely, we choose $\lambda = 8 > 0$. It turns out that the conclusions are similar to the ones found in the linear case ($\lambda = 0$). Qualitatively, the main difference is that during the time--evolution, the wave function spreads out more because of the additionally repulsive (defocusing) nonlinearity. In the linear case, the widest extension of the wave packet is always comparable to its final value. 
Choosing as before $\gamma_1 = 4\times 10^{-5}$ and $\gamma_2 = 1\times 10^{-9}$, we obtain the solution depicted in the right plot of Figure~\ref{fig:split_cond}, where we show the spatial density at the times $t=0$, $t=T=10$ and at the intermediate time $t=4$. The control is shown in the left plot.
In comparison to the linear case ($\lambda = 0$), the observable term in the objective functional $J(\psi, \alpha)$ is found to be slightly larger. Indeed, we obtain $$\langle A\psi(T),\psi(T)\rangle_{L_x^2}^2\approx 3.720 \times 10^{-3}.$$ 
This seems to indicate that nonlinear effects counteract the influence of the control potential.\\
\begin{figure}[ht]
  \centering
  \subfloat{\includegraphics[width=0.5\textwidth]{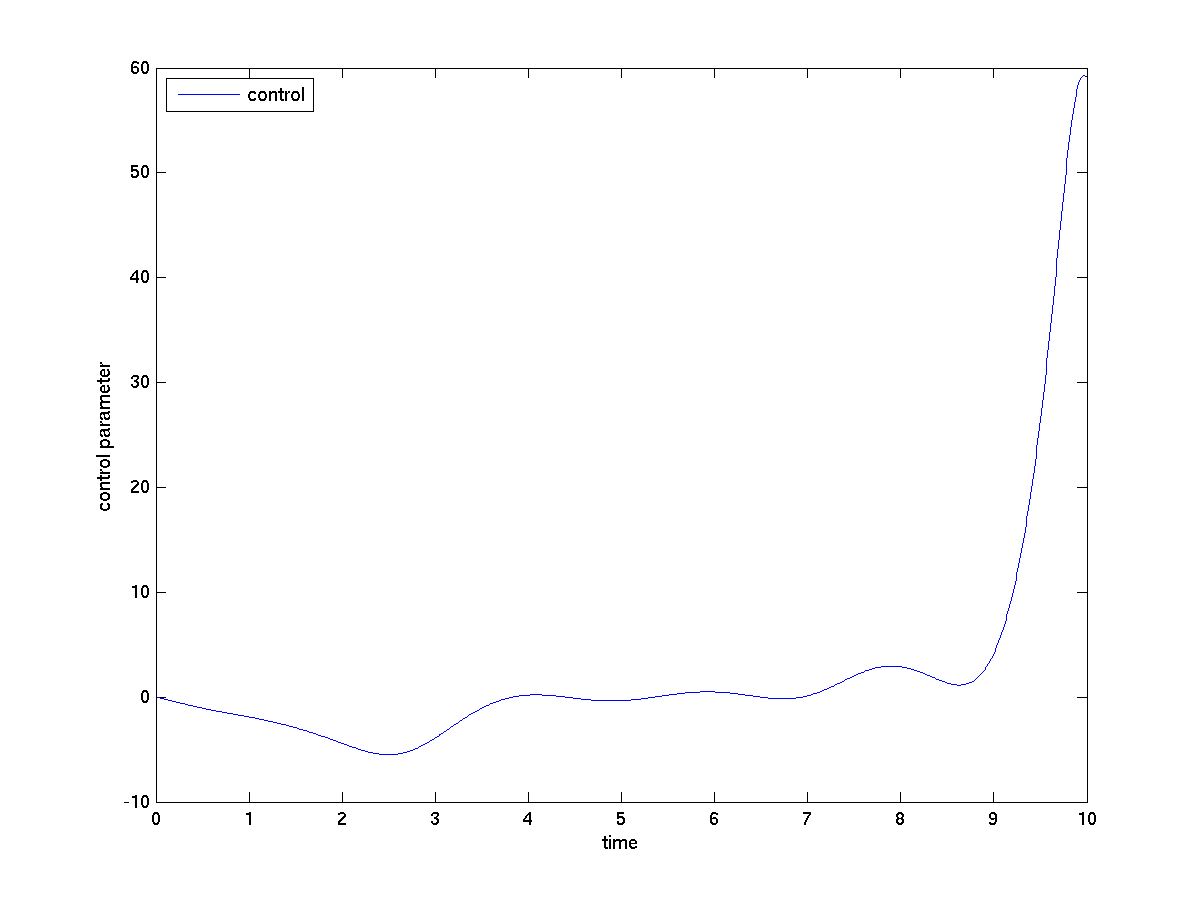}}
  \subfloat{\includegraphics[width=0.5\textwidth]{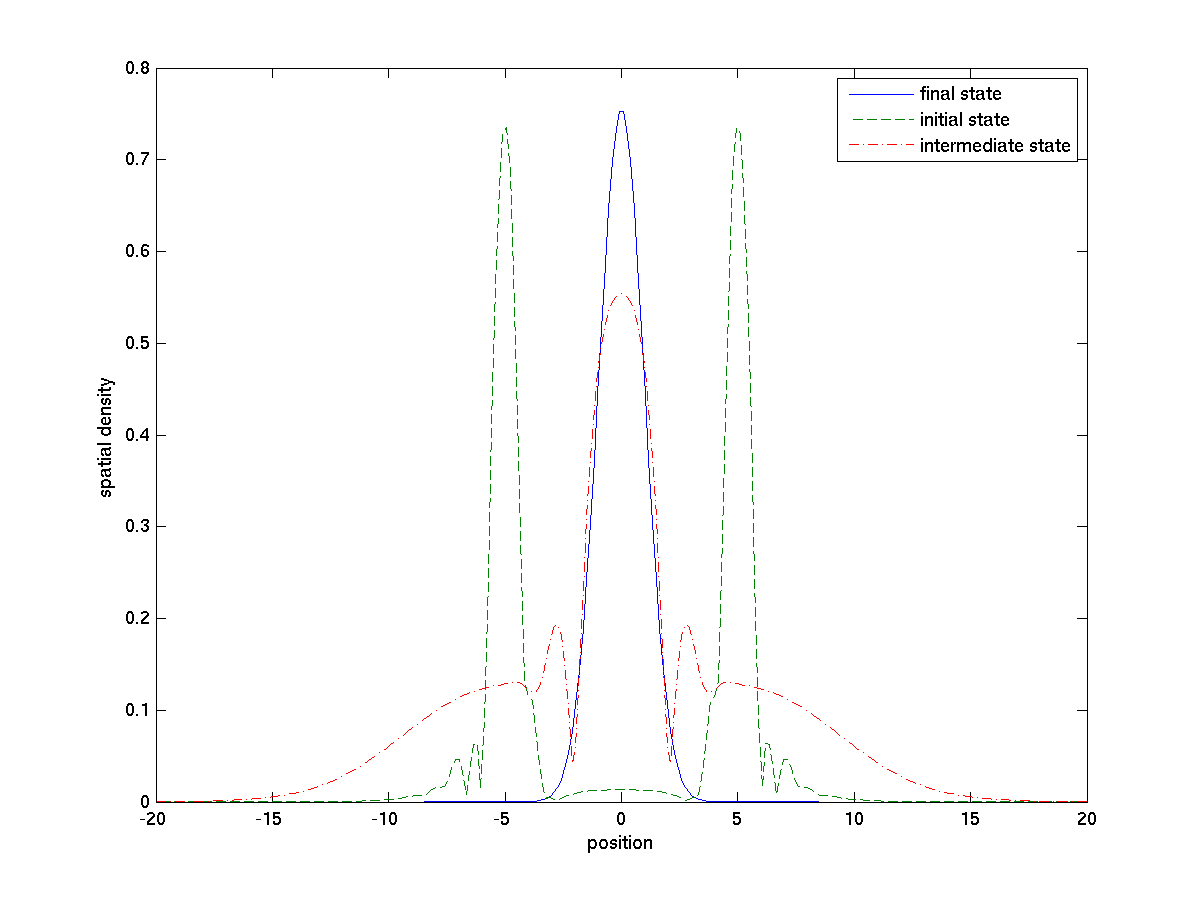}}
  \caption{Splitting a condensate with $\gamma_1 > 0$}
  \label{fig:split_cond}
\end{figure}

We again compare the present case with the one where $\gamma_1 = 0$ (i.e.\ no cost term proportional to the physical work) 
and $\gamma_2 = 1.5\times10^{-6}$. First, we find that $$\langle A\psi(T),\psi(T)\rangle_{L_x^2}^2\approx 3.382 \times 10^{-3}.$$ 
Moreover, $\|\dot{E}\|_{L_t^2}^2$ is about $150\%$ 
larger (172.1 versus 68.5) than in the case where $\gamma_1 \not =0$. Similarly, the total energy $\|E\|_{L_t^2}^2$ is around $15\%$ larger (91.8 versus 79.5). 

\subsubsection{Example: splitting an attractive Bose--Einstein condensate} 

Our numerical method allows us to go beyond the rigorous mathematical theory developed in the early chapters. In particular we may try to control the behavior of attractive condensates, which are modeled by \eqref{eq:schroed} with $\lambda < 0$, i.e.\ a focusing nonlinearity. Here we choose $\lambda=-1$, whereas the parameters $\gamma_1 = 4\times 10^{-5}$, $\gamma_2 = 1\times 10^{-9}$ are the same as before.
The results are shown in Figure~\ref{fig:focusing} (control in the left plot and the state at times $t=0,10,4$ in the right plot). The observable part of the objective functional satisfies \[\langle A\psi(T),\psi(T)\rangle_{L_x^2}^2\approx 2.143 \times 10^{-3}.\]\\
\begin{figure}[ht]
  \centering
  \subfloat{\includegraphics[width=0.5\textwidth]{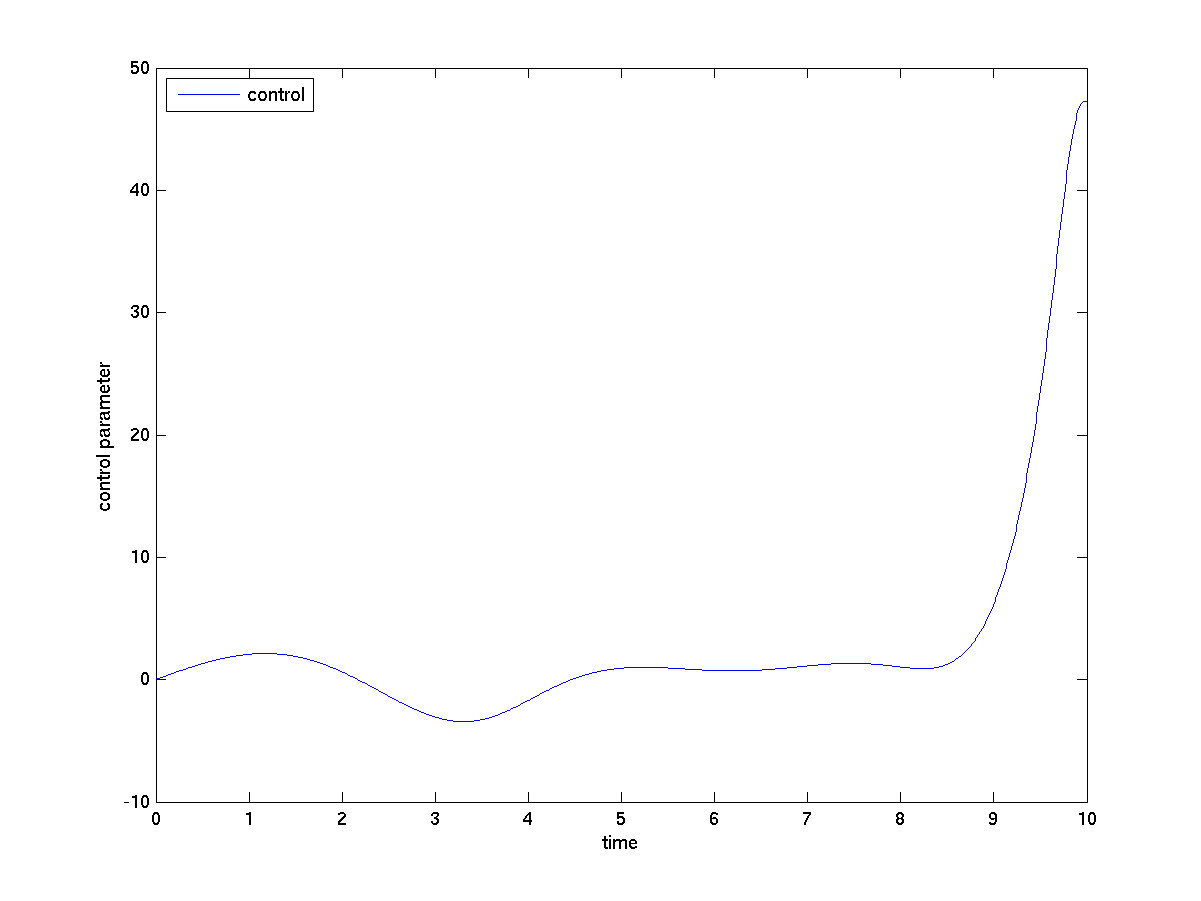}}
  \subfloat{\includegraphics[width=0.5\textwidth]{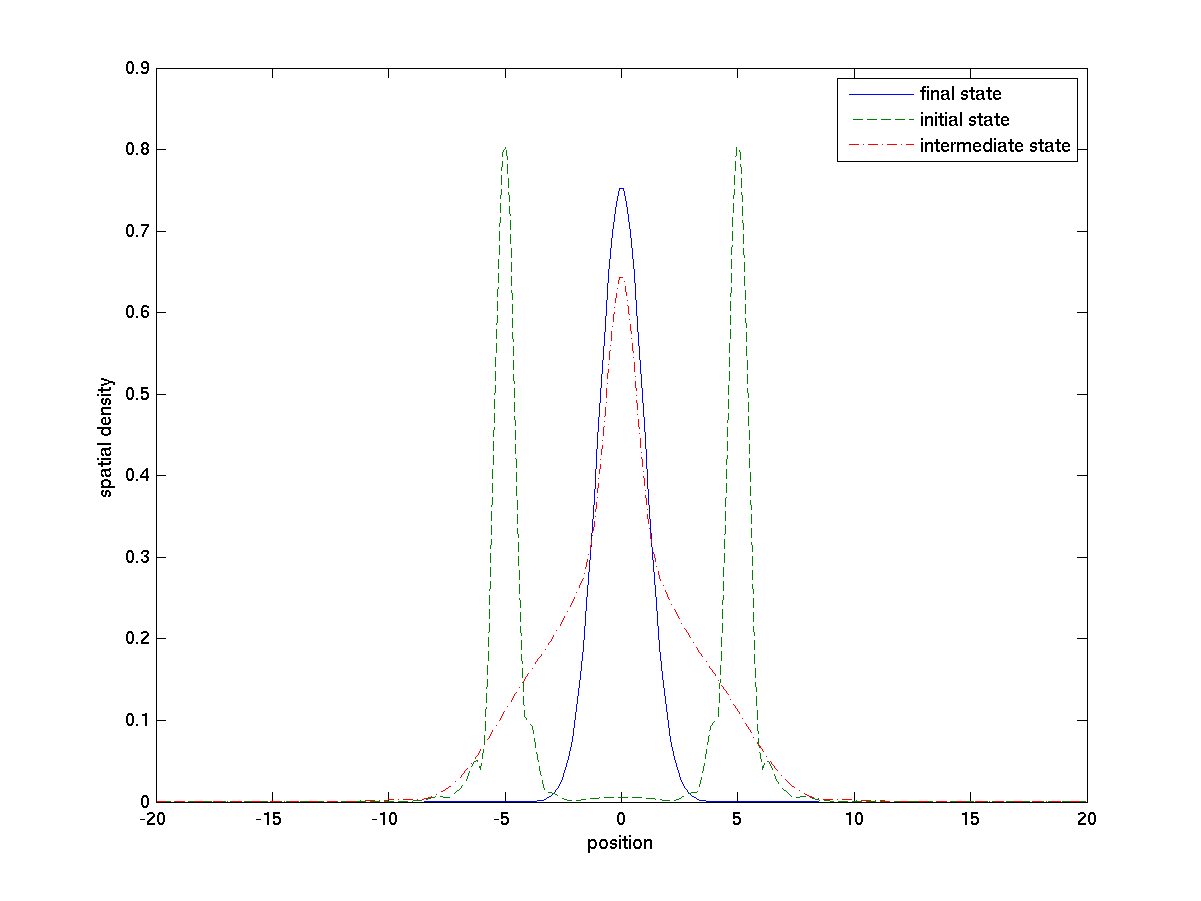}}
  \caption{Splitting a focusing condensate with $\gamma_1 > 0$}
  \label{fig:focusing}
\end{figure}

In comparison to the case of a repulsive (defocusing) nonlinearity the final value for the observable term $\langle A\psi(T),\psi(T)\rangle_{L_x^2}^2$ is much smaller, confirming the basic intuition that an attractive condensate does not tend to spread out as much as in the repulsive case.

\section{Concluding remarks}\label{sec:concl}

In this work, we have introduced a rigorous mathematical framework for optimal quantum control of linear and nonlinear Schr\"odinger equations of Gross-Pitaevskii type. 
We remark that in the physics literature, $L^2(\R^d)$ is usually considered as a complex Hilbert space, equipped with the inner product 
$\langle \varphi, \xi \rangle = \int_{\R^d} \varphi(x) \overline{\xi(x)} dx, $
whereas we consider $L^2(\R^d)$ as a real Hilbert space (of complex functions), equipped with \eqref{eq:inner_prod}. Note, however, that the expectation value 
of any physical observable $A$ and thus also $J(\psi, \alpha)$ is the same for both choices.

Let us briefly discuss possible generalizations for which our results remain valid. First, we point out that in our analysis above, we did not take advantage of the fact that $\gamma_1 > 0$ and hence all of our results remain true in the case $\gamma_1 = 0$. 
However, Example \ref{ex:splitlin} shows a significant quantitative difference in the behavior of the cost functionals with and without the term proportional to $\gamma_1$. 

Second, it is straightforward to extend our analysis to the case of {\it several} control parameters, i.e.
\[
{V}(t,x) = \sum_{k=1}^K \alpha_k(t) V_k(x), \quad K \ge 2.
\]
Clearly, for $V_k\in W^{m, \infty}(\R^d)$, $m\ge 2>d/2$, all of our results remain valid. In addition, it is not difficult to extend our framework to cases of 
{\it more general control potentials} $V(\alpha(t), x)$, not necessarily given in the form of a product. 
Such potentials are of physical significance; see cf.\ \cite{schm}. 
From the mathematical point of view, all of our results still apply provided that
$$\|V(\alpha, \cdot)\|_{W^{m,\infty}_x} \le C_1, \quad \|\partial^s_\alpha V(\alpha, \cdot)\|_{L^\infty_x} \le C_2, \ \forall |s| \le 2.$$ 
Note that in this case, the cost term in $J(\psi, \alpha)$, which is proportional to the physical work performed throughout the control process, reads
\[
 \int_0^T (\dot{E}(t))^2 dt  = \int_0^T (\dot{\alpha}(t))^2 \left( \int_{\R^d} \partial_\alpha V(\alpha(t),x) |\psi(t,x)|^2 \; dx \right)^2 dt.
\]

It is more problematic to provide a rigorous mathematical framework for control potentials $V(\alpha,x)$ which are {\it unbounded} with respect to $x\in \R^d$. 
Only in the case where $V(\alpha, x)$ is subquadratic with respect to $x$ and in $L^\infty(\R^d)$ with respect to $\alpha$, existence of a minimizer can be proved along the 
lines of the proof of Theorem~\ref{thm:ex_min}. More general unbounded control potentials $V(\alpha,x)$ definitely require new mathematical techniques. Note that in this case, even the existence of solutions to the nonlinear Schr\"odinger equation is not obvious.

Finally, we want to mention that it is possible to extend our results (with some technical effort) to the case of {\it focusing nonlinearities}, $\lambda < 0$, provided $\sigma < 2/d$. The latter prohibits 
the appearance of finite-time blow-up in the dynamics of the Gross--Pitaevskii equation. Clearly, the optimal control problem ceases to make sense if the solution to the 
underlying partial differential equation no longer exists.

\end{document}